\documentclass[a4paper,10pt]{article}

\usepackage[english]{babel}
\usepackage[T1]{fontenc}
\usepackage{graphicx}
\usepackage{amsmath, amsfonts,amsthm, mathrsfs, amssymb}
\usepackage{verbatim}
\usepackage{enumerate}
\usepackage[latin1]{inputenc}
\usepackage{hyperref}

\def\norma#1{\left\|#1\right\|}
\def\sleq{\preceq}

\newcommand{\C}{{\mathbb C}}

\newcommand{\N}{{\mathbb N}}

\newcommand{\R}{{\mathbb R}}

\newcommand{\T}{{\mathbb T}}
\newcommand{\Z}{{\mathbb Z}}

\newcommand{\cF}{{\mathcal F}}

\newcommand{\cH}{{\mathcal H}}

\newcommand{\cJ}{{\mathcal J}}

\newcommand{\cO}{{\mathcal O}}
\newcommand{\cP}{{\mathcal P}}

\newcommand{\cR}{{\mathcal R}}
\newcommand{\cS}{{\mathcal S}}
\newcommand{\cT}{{\mathcal T}}
\newcommand{\cU}{{\mathcal U}}

\newcommand{\cW}{{\mathcal W}}

\newcommand{\cZ}{{\mathcal Z}}

\renewcommand{\d}{\partial}

\newcommand{\grad}{\nabla}

\newcommand{\norm}[1]{\| #1 \|}

\newcommand{\la}{\left\langle}
\newcommand{\ra}{\right\rangle}

\newcommand{\jap}[1]{\langle #1 \rangle}
\def\di{{\rm d}}

\newcommand{\nablac}{\jap{\grad}_c}
\newcommand{\sHc}{\mathscr{H}_c}
\newcommand{\sWc}{\mathscr{W}_c}

\newtheorem{theorem}{Theorem}[section]
\newtheorem{lemma}[theorem]{Lemma}
\newtheorem{corollary}[theorem]{Corollary}
\newtheorem{proposition}[theorem]{Proposition}
\newtheorem{definition}[theorem]{Definition}
\newtheorem{remark}[theorem]{Remark}

\title{Dynamics of the nonlinear Klein-Gordon equation in the nonrelativistic limit, I} 

\author{
 S. Pasquali \newline
\footnote{ \textit{Email: } \texttt{stefano.pasquali@unimi.it} }
}

\begin{document}

\maketitle

\begin{abstract}
The nonlinear Klein-Gordon (NLKG) equation on a manifold $M$ in the 
nonrelativistic limit, namely as the speed of light $c$ tends to 
infinity, is considered. 
In particular, a higher-order normalized approximation of NLKG 
(which corresponds to the NLS at order $r=1$) is constructed, 
and when $M$ is a smooth compact manifold or $\mathbb{R}^d$ it is proved that 
the solution of the approximating equation approximates the solution of the 
NLKG locally uniformly in time. 
When $M=\mathbb{R}^d$, $d \geq 2$, it is proved that solutions of the linearized
 order $r$ normalized equation approximate solutions of linear Klein-Gordon 
equation up to times of order $\mathcal{O}(c^{2(r-1)})$ for any $r>1$. 
\end{abstract}


\section{Introduction} \label{intro}

In this paper the nonlinear Klein-Gordon (NLKG) equation 
in the nonrelativistic limit, namely as the speed of light $c$ 
tends to infinity, is studied. 
Formal computations going back to the the first half of the
last century suggest that, up to corrections of order $\cO(c^{-2})$, the
system should be described by the nonlinear Schr\"odinger (NLS) 
equation. Subsequent mathematical results have shown that the NLS 
describes the dynamics over time scales of order $\cO(1)$.

The nonrelativistic limit for the Klein-Gordon equation 
on $\R^d$ has been extensively studied over more then 30 years, and
essentially all the known results only show convergence of the
solutions of NLKG to the solutions of the approximate equation for
times of order $\cO(1)$. The typical statement ensures convergence locally
uniformly in time. In a first series of results (see
\cite{tsutsumi1984nonrelativistic}, \cite{najman1990nonrelativistic} and 
\cite{machihara2001nonrelativistic}) it was shown that, 
if the initial data are in a certain smoothness class, 
then the solutions converge in a weaker
topology to the solutions of the approximating equation. These are
informally called ``results with loss of smoothness''. Although  
in this paper a longer time convergence is proved, these results also fill in 
this group.

Some other results, essentially due to Machihara, Masmoudi, 
Nakanishi and Ozawa, ensure convergence without loss of regularity
in the energy space, again over time scales of order $\cO(1)$ (see 
\cite{machihara2002nonrelativistic}, \cite{masmoudi2002nonlinear} and \cite{nakanishi2008transfer}).

Concerning radiation solutions there is a remarkable result 
(see \cite{nakanishi2002nonrelativistic}) by Nakanishi, who considered the complex NLKG 
in the defocusing case, in which it is known that all solutions scatter 
(and thus the scattering operator exists), and proved that the scattering 
operator of the NLKG equation converges to the scattering operator of the NLS. 
It is important to remark that this result is not contained in the one proved 
here and does not contain it.

Recently Lu and Zhang in \cite{lu2016partially} proved a result 
which concerns the NLKG with a quadratic nonlinearity. 
Here the problem is that the typical scale
over which the standard approach allows to control the dynamics is
$\cO(c^{-1})$, while the dynamics of the approximating equation takes 
place over time scales of order $\cO(1)$. In that work the authors are 
able to use a normal form transformation (in a spirit quite different 
from ours) in order to extend the time of validity of the approximation 
over the $\cO(1)$ time scale. 
We did not try to reproduce or extend that result. \\

In this paper some results for the dynamics of NLKG are obtained. 
Actually two kinds of results are proved: a global 
existence result for NLKG (see Theorem \ref{GlobExNLKG}), 
uniform as $c\to\infty$, and approximation results 
(see Theorem \ref{introlocuniftconv} and Theorem \ref{KGtoLSradr}) 
showing that solutions of NLKG can be approximated 
by solutions of suitable higher order NLS equations. 
Approximation results are different in the case where 
the equation lives on $\R^3$ or in a compact manifold: 
when $M$ is a smooth compact manifold or $\R^d$ 
the solution of NLS approximates the solution of the original equation 
locally uniformly in time; when $M=\R^d$, $d \geq 2$, it is possible to prove 
that solutions of the \emph{linearized} approximating equation approximate 
solutions of the linear Klein-Gordon equation up to times of order 
$\cO(c^{2(r-1)})$, for any $r>1$. \\

The present paper can be thought as an example in which techniques from 
canonical perturbation theory are used together with results from 
the theory of dispersive equations in order to understand the singular limit 
of some Hamiltonian PDEs. 
In this context, the nonrelativistic limit of the NLKG is a relevant example.

The issue of nonrelativistic limit has been studied 
also in the more general Maxwell-Klein-Gordon system 
(\cite{bechouche2004nonrelativistic}, \cite{masmoudi2003nonrelativistic}), 
in the Klein-Gordon-Zakharov system (\cite{masmoudi2008energy}, \cite{masmoudi2010klein}), 
in the Hartree equation (\cite{cho2006semirelativistic}) and in the pseudo-relativistic 
NLS (\cite{choi2016nonrelativistic}). However, all these results proved 
the convergence of the solutions of the limiting system in the energy space 
(\cite{cho2006semirelativistic} studied also the convergence in $H^k$), 
\emph{locally uniformly in time}; 
no information could be obtained about the convergence of solutions for 
longer (in the case of NLKG, that means $c$-dependent) timescales. \\

Other examples of singular perturbation problems that have been studied either
with canonical perturbation theory or with other techniques (typically
multiscale analysis) are the problem of the continuous approximation
of lattice dynamics (see e.g. \cite{bambusi2006metastability}, \cite{schneider2010bounds}) 
and the semiclassical analysis of Schr\"odinger operators 
(see e.g. \cite{robert1987semi}, \cite{alazard2007semi}). 
In the framework of lattice dynamics, the time scale covered by all known 
results is that typical of averaging theorems, which corresponds to our 
$\cO(1)$ time scale. Hopefully the methods developed in the present paper 
could allow to extend the time of validity of those results. \\

The paper is organized as follows.
In sect. \ref{results} we state the results of the paper, together with 
some examples and comments. 
In sect. \ref{dispKG} we show Strichartz estimates for the linear KG 
equation and for the KG equation with potential, as well as a global existence 
result uniform with respect to $c$ for the cubic NLKG equation on $\R^3$. 
In sect. \ref{Galavmethod} we state the main abstract result of the paper. 
In the subsequent sect. \ref{Galavproof} we present the proof of the abstract 
result, which is based on a Galerkin averaging technique, along with some 
remarks and variant of the result. 
Next, in sect. \ref{NLKGappl} we apply the abstract theorem to the NLKG 
equation, making some explicit computations of the normal form at the first 
and at the second step. In the following sect. \ref{dynamics} we deduce 
some results about the approximation of solutions locally uniformly in time, 
while in sect. \ref{longtappr} we discuss the approximation for longer 
timescales: in particular, to deduce the latter we will exploit 
some dispersive properties of the KG equation reported in sect. \ref{dispKG}.
Finally, in Appendix \ref{BNFest} we will report some Birkhoff Normal Form 
estimates (the approach is essentially the same as in \cite{bambusi1999nekhoroshev}), 
and in Appendix \ref{interp} we will prove some interpolation theory results 
for relativistic Sobolev spaces, and we exploit them to deduce Strichartz 
estimates for the KG equation with potential. \\

\emph{Acknowledgments}.
This work is based on author's PhD thesis. 
He would like to express his thanks to his supervisor Professor Dario Bambusi.

\section{Statement of the Main Results} \label{results}

The NLKG equation describes the motion of a spinless particle with mass $m>0$. 
Consider first the real NLKG
\begin{align} \label{NLKGeq}
\frac{\hbar^2}{2mc^2} u_{tt} - \frac{\hbar^2}{2m} \Delta u +\frac{mc^2}{2} u + \lambda |u|^{2(l-1)}u &= 0,
\end{align}
where $c>0$ is the speed of light, $\hbar>0$ is the Planck constant, 
$\lambda \in \R$, $l \geq 2$, $c>0$.

In the following $m=1$, $\hbar=1$. 
As anticipated above, one is interested in the behaviour of solutions
as $c\to\infty$. 

First it is convenient to reduce equation
\eqref{NLKGeq} to a first order system, by making the following
symplectic change variables 
\begin{align*}
 \psi &:= \frac{1}{\sqrt{2}} \left[ \left(\frac{\nablac}{c} \right)^{1/2} u - i \left(\frac{c}{\nablac}\right)^{1/2}v \right], \; \; v = u_t/c^2,
\end{align*}
where
\begin{equation}
\label{nablac}
\nablac:=(c^2-\Delta)^{1/2}, 
\end{equation}
which reduces \eqref{NLKGeq} to the form 
\begin{align} \label{dsa}
-i \psi_t &= c \nablac \psi 
+ \frac{\lambda}{2^l} \left( \frac{c}{\nablac} \right)^{1/2} 
\left[ \left( \frac{c}{\nablac} \right)^{1/2} (\psi+\bar\psi) \right]^{2l-1},
\end{align}
which is hamiltonian with Hamiltonian function given by
\begin{align} \label{dsa1}
H(\bar\psi,\psi) &= \la \bar{\psi}, c\jap{\grad}_c\psi \ra 
+ \frac{\lambda}{2l} \int \left[ 
\left( \frac{c}{\nablac} \right)^{1/2} \frac{\psi+\bar\psi}{\sqrt{2}} 
\right]^{2l} \di x.
\end{align}

To state our first result, introduce for any $k \in \R$ and 
for any $1 < p < \infty$ the following 
relativistic Sobolev spaces
\begin{align} \label{relsobspace}
\sWc^{k,p}(\R^3) &:= \left\{ u\in L^p: \|u\|_{\sWc^{k,p}}:=\norma{c^{-k} \, \nablac^k  u}_{L^p}<+\infty \right\}, \\
\sHc^{k}(\R^3) &:= \left\{ u\in L^2: \|u\|_{\sHc^k}:=\norma{c^{-k} \, \nablac^k  u}_{L^2}<+\infty \right\},
\end{align}
and remark that the energy space is $\sHc^{1/2}$. Remark that 
for finite $c>0$ such spaces coincide with the standard Sobolev spaces, 
while for $c=\infty$ they are equivalent to the Lebesgue spaces $L^p$.

\indent In the following the notation $a \sleq b$ is used to mean: 
there exists a positive constant $K$ that does not depend on $c$ such that $a \leq Kb$.\\

We begin with a global existence result for the NLKG \eqref{dsa} 
in the cubic case, $l=2$, for small initial data.

\begin{theorem} \label{GlobExNLKG}
Consider Eq. \eqref{dsa} with $l=2$ on $\R^3$. \\
There exists $\epsilon_\ast>0$ such that, if the norm of the initial datum
$\psi_0$ fulfills
\begin{align}
\norma{\psi_0}_{\sHc^{1/2}} &\leq \epsilon_\ast,
\end{align}
then the corresponding solution $\psi(t)$ of \eqref{dsa} exists 
globally in time:
\begin{align}
\| \psi(t) \|_{L^\infty_t \sHc^{1/2} } \; &\sleq \; \| \psi_0 \|_{\sHc^{1/2}}, 
\end{align}
All the constants do not depend on $c$.
\end{theorem}

\begin{remark} \label{r.1}
For finite $c$ this is the standard result for small
amplitude solution, while for $c=\infty$ it becomes the standard
result for the NLS: thus Theorem \ref{GlobExNLKG} interpolates
between these apparently completely different situations. 
Remark that the lack of a priori estimates for the solutions 
of NLKG in the limit $c\to\infty$ was the main obstruction in order to 
obtain global existence results uniform in $c$ in standard Sobolev spaces.
\end{remark}

One is now interested in discussing the approximation of the solutions
of NLKG with NLS-type equations. Before giving the result we
describe the general strategy we use to get them. 

Remark that Eq. \eqref{NLKGeq} is Hamiltonian with Hamiltonian function
\eqref{dsa1}. If one divides the Hamiltonian by a factor $c^2$ 
(which corresponds to a rescaling of time) 
and expands in powers of $c^{-2}$ it takes the form 
\begin{equation} \label{Hmi}
\langle\psi,\bar \psi\rangle + \frac{1}{c^2} P_c(\psi,\bar \psi)
\end{equation}
with a suitable funtion $P_c$. 
One can notice that this Hamiltonian is a perturbation of 
$h_0:=\langle\psi,\bar \psi\rangle $,
which is the generator of the standard Gauge transform, and 
which in particular admits a flow that is periodic in time. 
Thus the idea is to exploit canonical perturbation theory in
order to conjugate such a Hamiltonian system to a system in normal
form, up to remainders of order $\cO(c^{-2r})$, for any given $r \geq 1$. 

The problem is that the perturbation $P_c$ has a vector field which 
is small only as an operator extracting derivatives. One can Taylor expand 
$P_c$ and its vector field, but the number of derivatives extracted at
each order increases. This situation is typical in singular perturbation
problems. Problems of this kind have already been studied with
canonical perturbation theory, but the price to pay to get a normal
form is that the remainder of the perturbation turns out to be an
operator that extracts a large number of derivatives. 

In Sect. \ref{NLKGappl} the normal form equation is explicitly computed 
in the case $r=2$:
\begin{align} \label{fdas2}
-i \psi_t \; &= \;  c^2 \psi - \frac{1}{2} \Delta\psi + \frac{3}{4} \lambda |\psi|^2\psi \nonumber \\
&+ \frac{1}{c^2} \left[ \frac{51}{8} \lambda^2 |\psi|^4\psi + \frac{3}{16} \lambda \left(2|\psi|^2 \, \Delta\psi + \psi^2 \Delta\bar\psi + \Delta(|\psi|^2\bar\psi) \right) - \frac{1}{8} \Delta^2\psi \right],
\end{align}
namely a singular perturbation of a Gauge-transformed
NLS equation. If one, after a gauge transformation, 
only considers the first order terms, 
one has the NLS, for which radiation solution exist 
(for example in the defocusing case all solutions are of radiation type). 
For higher order NLS there are very few results 
(see for example \cite{maeda2011existence}). 

The standard way to exploit such a ``singular'' normal form is to use it
just to construct some approximate solution of the original system, and
then to apply Gronwall Lemma in order to estimate the difference with
a true solution with the same initial datum (see for example 
\cite{bambusi2002nonlinear}).

This strategy works also here, but it only leads to a control of the
solutions over times of order $\cO(c^2)$. When scaled back to the
physical time, this allows to justify the approximation of the solutions 
of NLKG by solutions of the NLS over time scales of order $\cO(1)$, 
{\it on any manifold} admitting a Littlewood-Paley decomposition 
(such as Riemannian smooth compact manifolds, or $\R^d$; 
see the introduction of \cite{bouclet2010littlewood} for the construction 
of Littlewood-Paley decomposition on manifolds). 

\begin{theorem} \label{introlocuniftconv}
Let $M$ be a manifold which admits a Littlewood-Paley decomposition, 
and consider Eq. \eqref{dsa} on $M$.

Fix $r \geq 1$, $R>0$, $k_1 \gg 1$, $1<p< +\infty$. 
Then $\exists$ $k_0=k_0(r)>0$ with the following properties: 
for any $k \geq k_1$ there exists $c_{l,r,k,p,R} \gg 1$ such that 
for any $c>c_{l,r,k,p,R}$, if 
\begin{align*}
\|\psi_0\|_{k+k_0,p} &\leq R
\end{align*}
and there exists $T=T_{r,k,p}>0$ such that the solution $\psi_{r}$ 
of the equation in normal form up to order $r$ \eqref{simpleq} 
with initial datum $\psi_0$ satisfies
\begin{align*}
\|\psi_r(t)\|_{k+k_0,p} & \leq 2R, \; \; \text{for} \; \; 0 \leq t \leq T,
\end{align*}
then
\begin{align} 
\|\psi(t)-\psi_r(t)\|_{k,p} &\sleq \frac{1}{c^2},\; \; \text{for} \; \; 0\leq t \leq T.
\end{align}
where $\psi(t)$ is the solution of \eqref{dsa} with initial datum $\psi_{0}$.
\end{theorem}

A similar result has been obtained for the case $M=\T^d$ by Faou and Schratz, 
who aimed to construct numerical schemes which are robust 
in the nonrelativistic limit (see \cite{faou2014asymptotic}; see also  
\cite{bao2012analysis}, \cite{bao2016uniformly} and to 
\cite{baumstark2016uniformly} for some numerical analysis of the 
nonrelativistic limit of the NLKG).

The idea one uses here in order to improve the time scale of the result
is that of substituting Gronwall Lemma with a more sophisticated tool,
namely dispersive estimates and the retarded Strichartz estimate. 
This can be done each time one can prove a dispersive or a Strichartz estimate 
(in the spaces $\sWc^{k,p}$ or $W^{k,p}$) for the linearization of equation 
\eqref{dsa} on the approximate solution uniformly in $c$.

It turns out that this is a quite hard task, and we were able to accomplish it 
only for the linear KG equation on $\R^d$. 
In order to state our approximation result, we consider the approximate 
equation given by the Hamilton equations 
of the normal form truncated at order $\cO(c^{-2r})$, and 
let $\psi_r$ be a solution of such a linearized normal form equation. 

\begin{theorem} \label{KGtoLSradr}
Consider \eqref{NLKGeq} on $\R^d$, $d \geq 2$. 
Fix $r \geq 1$ and $k_1 \gg 1$.
Then $\exists$ $k_0=k_0(r)>0$ such that for any $k \geq k_1$, 
if we denote by $\psi_r$ the solution of the linearized normal equation 
\eqref{schrordr} with initial datum $\psi_0 \in H^{k+k_0}$ and by $\psi$ 
the solution of the linear KG equation \eqref{KG} with the same initial datum, 
then there exists $c^\ast:=c^\ast(r,k) > 0$ such that for any $c > c^\ast$ 
\begin{align*}
\sup_{t\in [0,T]} \|\psi(t)-\psi_r(t)\|_{H^k_x} &\sleq \frac{1}{c^2}, \; \; T \sleq c^{2(r-1)}.
\end{align*}
\end{theorem}

This result has been proved in the case $r=1$ in Appendix A of 
\cite{carles2015higher}.

In order to approximate small radiation solutions of the NLKG equation, 
we would need to use dispersive estimates for the normal form equation, 
which unfortunately are not present in the literature. 
We defer this problem to a future work. \\

There are other well known solutions of NLS which would be interesting to 
study; indeed, it is well known that in the case of mixed-type nonlinearity
\begin{align*}
i \psi_t &=-\Delta \psi - (|\psi|^2-|\psi|^4)\psi,
\end{align*}
such an equation admits linearly stable solitary wave solutions; it can also 
be proved that the standing waves of NLS can be modified in order to
obtain standing wave solutions of the normal form of order $r$, for any $r$. 
It would be of clear interest to prove that true solutions starting close
to such standing wave remain close to them for long times (remark that the 
NLKG does not admit stable standing wave solutions, see \cite{ohta2007strong}); 
in order to get such a result one should prove a Strichartz estimate for NLKG 
close to the approximate solution and uniformly in $c$. \\

Before closing the subsection, a few technical comments: the
first one is that here we develop normal form in the framework of the spaces 
$W^{k,p}$, while known results in Galerkin averaging theory only allow to deal 
with the spaces $H^k$. This is due to the fact that the Fourier analysis is 
used in order to approximate the derivatives operators with bounded
operators. Thus the first technical step needed in order to be able to
exploit dispersion is to reformulate Galerkin averaging theory in terms
of dyadic decompositions. This is done in Theorem \ref{normformgavthm}. \\

\section{Dispersive properties of the Klein-Gordon equation} \label{dispKG}

\indent We briefly recall some classical notion of Fourier 
analysis on $\R^d$. Recall the definition of the space of 
Schwartz (or rapidly decreasing) functions,

\begin{align*}
\cS &:= \{ f \in C^\infty(\R^d,\R) | \sup_{x \in \R^d} (1+|x|^2)^{\alpha/2} |\d^\beta f(x)| < + \infty, \; \; \forall \alpha \in \N^d, \forall \beta \in \N^d \}.
\end{align*}

In the following $\la x \ra:=(1+|x|^2)^{1/2}$. \\
Now, for any $f \in \cS$ the \emph{Fourier transform} of $f$, 
$\hat f:\R^d \to \R$, is defined by the following formula

\begin{align*}
\hat f(\xi) &:= (2\pi)^{-d/2} \int_{\R^d} f(x) e^{-i \la x,\xi \ra}\di x, \; \; \forall \xi \in \R^d,
\end{align*}
where $\la \cdot,\cdot \ra$ denotes the scalar product in $\R^d$.

At the beginning we will obtain Strichartz estimates for the linear equation 
\begin{align} \label{KG}
-i \, \psi_t \, &= \, c\jap{\grad}_c \, \psi, \; \; x \in \R^3. 
\end{align}

\begin{proposition} \label{strlin}
For any Schr\"odinger admissible couples $(p,q)$ and $(r,s)$, namely such that \\
\begin{align*}
2 \leq p&,r \leq \infty, \\ 
2 \leq q&,s \leq 6, \\ 
\frac{2}{p}+\frac{3}{q} =\frac{3}{2}, &\; 
\frac{2}{r}+\frac{3}{s} =\frac{3}{2},
\end{align*}
one has
\begin{align} \label{strestkg}
\|  \jap{\grad}_c^{\frac{1}{q}-\frac{1}{p}} \; e^{it \; c\jap{\grad}_c} \; \psi_0 \|_{L^p_t L^q_x} \; &\sleq \; c^{\frac{1}{q}-\frac{1}{p}-\frac{1}{2}} \; \| \jap{\grad}_c^{1/2}  \psi_0\|_{L^2}, 
\end{align}
\begin{align} \label{retstrkg}
\left\|  \jap{\grad}_c^{\frac{1}{q}-\frac{1}{p}} \; \int_0^t e^{i(t-s) \; c\jap{\grad}_c} \; F(s) \; \di s \right\|_{L^p_t L^q_x} \; &\sleq \; c^{\frac{1}{q}-\frac{1}{p}+\frac{1}{s}-\frac{1}{r}-1} \; \| \jap{\grad}_c^{\frac{1}{r}-\frac{1}{s}+1} F \|_{L^{r'}_t L^{s'}_x}. 
\end{align}
\end{proposition}

\begin{remark}
The above result can be easily generalized to the $d$-dimensional case, 
$d \geq 2$, if we consider $(p,q)$ and $(r,s)$ such that
\begin{align*}
2 \leq p&,r \leq \infty, \\ 
2 \leq q&,s \leq \frac{2d}{d-2}, \\ 
\frac{2}{p}+\frac{d}{q} =\frac{d}{2}, &\; 
\frac{2}{r}+\frac{d}{s} =\frac{d}{2}, \\
(p,q,d),(r,s,d) &\neq (2,+\infty,2),
\end{align*}
\end{remark}

\begin{remark}
By choosing $p=+\infty$ and $q=2$, we get the following a priori estimate for 
finite energy solutions of \eqref{KG},
\begin{align*}
\| c^{1/2} \jap{\grad}_c^{1/2} \; e^{it \; c\jap{\grad}_c} \; \psi_0 \|_{L^\infty_t L^2_x} \; &\sleq \; \| c^{1/2} \jap{\grad}_c^{1/2}  \psi_0\|_{L^2}.
\end{align*}
We also point out that, since the operators $\jap{\grad}$ and $\jap{\grad}_c$ 
commute, the above estimates in the spaces $L^p_tL^q_x$ extend to 
estimates in $L^p_tW^{k,q}_x$ for any $k \geq 0$.
\end{remark}

\begin{proof}
We recall a result reported by D'Ancona-Fanelli in \cite{d2008strichartz} for the operator $\jap{\grad}:=\jap{\grad}_1$.

\begin{lemma} \label{danfanlemma}
For all $(p,q)$ Schr\"odinger-admissible exponents (ie, s.t. $\frac{2}{p}+\frac{3}{q}=\frac{3}{2}$) \\
\[ \|e^{i\tau \; \jap{\grad} } \; \phi_0 \|_{L^p_\tau \; W^{\frac{1}{q}-\frac{1}{p}-\frac{1}{2},q }_y} = \; \|\jap{\grad}^{\frac{1}{q}-\frac{1}{p}-\frac{1}{2} } \; e^{it \; \jap{\grad}} \; \phi_0 \|_{L^p_\tau \; L^q_y} \; \leq \; \|\phi_0\|_{L^2_y}. \] \\
\end{lemma}

Now, the solution of equation \eqref{KG} satifies 
$\hat\psi(t,\xi) = e^{ i c \jap{\xi}_c t}\hat\psi_0(\xi)$. 
We then define $\eta:= \xi/c$, in order to have that 
\begin{align*}
\hat\phi(c^2 t, \eta) &:= \hat\psi(t, c\eta) = \hat\psi(t, \xi), \\
\end{align*}
and in particular that $\hat\phi_0(\eta) = \hat\psi_0(\xi)$. \\
Since 
\begin{align}
\la\xi\ra_c = \sqrt{c^2+|\xi|^2} = c \sqrt{1+|\xi|^2/c^2},
\end{align}
we get 
\begin{align*}
\hat\phi(t, \eta) &= e^{it \, c^2 \jap{\xi/c} } \hat\phi_0(\xi/c) \\
&= e^{ i \, tc^2 \, \jap{\eta} } \hat\phi_0(\eta) \\
&= e^{ i \, \tau \, \jap{\eta} } \hat\phi_0(\eta)
\end{align*}
if we set $\tau:=c^2t$. Now, by setting $y:=cx$ a simple scaling argument leads to \\
\[ \|e^{i\tau \; \jap{\grad} } \; \phi_0 \|_{L^p_\tau \; L^q_y} \; \sleq \; \|\jap{\grad}^{\frac{1}{p}-\frac{1}{q}+\frac{1}{2} } \; \phi_0 \|_{L^2}  \; = \; \| \la\eta\ra^{\frac{1}{p}-\frac{1}{q}+\frac{1}{2} } \hat\phi_0\|_{L^2} \]
and since 
\begin{align*}
\| \la\eta\ra^{k} \hat\phi_0\|^2_{L^2} \; &= \; \int_{\R^3} \la\eta\ra^{2k} \; |\hat\phi_0(\eta)|^2 \; \di \eta \\
&= \; \int_{\R^3} \la \frac{\xi}{c} \ra^{2k} \; |\hat\phi_0(\eta/c)|^2 \; \frac{\di \xi}{c^3} \; = \; \frac{1}{c^{2k+3}} \; \int_{\R^3} \la\xi\ra_c^{2k} \; |\hat\psi_0(\xi)|^2 \; \di \xi,
\end{align*}
we get 
\begin{align}
\| \la\eta\ra^{\frac{1}{p}-\frac{1}{q}+\frac{1}{2} } \hat\phi_0\|_{L^2} \; &= \; \frac{1}{c^{\frac{3}{2}-\frac{1}{q}+\frac{1}{p}+\frac{1}{2} } } \; \| \jap{\grad}_c^{ \frac{1}{p}-\frac{1}{q}+\frac{1}{2} } \; \psi_0\|_{L^2},
\end{align}
while on the other hand 
\begin{align*}
\psi(t, x) \; &= (2\pi)^{-d/2} \int_{\R^3} e^{i \la\xi,x\ra} \; \hat\psi(t, \xi) \; \di \xi \; 
= (2\pi)^{-d/2} \int_{\R^3} e^{i \la\eta,cx\ra} \; \hat\psi(t, c\eta) \; c^3 \di\eta \\
&= (2\pi)^{-d/2} \; c^3 \; \int_{\R^3} e^{i \la\eta,cx\ra} \; \hat\phi(c^2t, \eta) \;  \di \eta \; 
= c^3 \; \phi(c^2t, cx),
\end{align*}
yields 
\begin{align}
\|\psi\|_{L^p_t L^q_x} \; = \; c^{3 - \; 3/q - \; 2/p} \; \|\phi\|_{L^p_\tau L^q_y}. 
\end{align}
Hence we can deduce \eqref{strestkg}; via a scaling argument we can also deduce \eqref{retstrkg}.
\end{proof}

One important application of the Strichartz estimates for the free 
Klein-Gordon equation is Theorem \ref{GlobExNLKG}, namely a global 
existence result \emph{uniform with respect to c} for the NLKG equation 
\eqref{dsa} with cubic nonlinearity (this means $l=2$), with small initial data.

\begin{proof}[Proof of Theorem \ref{GlobExNLKG}]
It just suffices to apply Duhamel formula,
\begin{align*}
\psi(t) &= e^{it c \nabla_c}\psi_0 + i \frac{\lambda}{2^l} \int_0^t e^{i(t-s) c \nabla_c} \left( \frac{c}{\nablac} \right)^{1/2} 
\left[ \left( \frac{c}{\nablac} \right)^{1/2} (\psi+\bar\psi) \right]^{2l-1},
\end{align*}
and Proposition \ref{strlin} with $p=+\infty$, in order to get that
\begin{align*}
\| \psi(t) \|_{L^\infty_t \sHc^{1/2} } &\sleq \| \psi_0 \|_{\sHc^{1/2}} + c^{1/s - 1/r} \norma{ \nabla_c^{1/r - 1/s} \left[ \left( \frac{c}{\nablac} \right)^{1/2} (\psi+\bar\psi) \right]^{3} }_{L^{r'}_t L^{s'}_x},
\end{align*}
but by choosing $r=+\infty$ and by H\"older inequality we get
\begin{align*}
\| \psi(t) \|_{L^\infty_t \sHc^{1/2} } &\sleq \| \psi_0 \|_{\sHc^{1/2}} + 
\norma{ \left[ \left( \frac{c}{\nablac} \right)^{1/2} (\psi+\bar\psi) \right]^{3} }_{L^1_t L^2_x} \\
&\sleq \| \psi_0 \|_{\sHc^{1/2}} + \norma{ \left[ \left( \frac{c}{\nablac} \right)^{1/2} (\psi+\bar\psi) \right]^2 }_{L^1_t L^3_x}  \norma{ \left( \frac{c}{\nablac} \right)^{1/2} (\psi+\bar\psi)  }_{L^\infty_t L^6_x} \\
&\sleq \| \psi_0 \|_{\sHc^{1/2}} + \norma{ \left( \frac{c}{\nablac} \right)^{1/2} (\psi+\bar\psi)  }^2_{L^2_t L^6_x}  \norma{ \left( \frac{c}{\nablac} \right)^{1/2} (\psi+\bar\psi)  }_{L^\infty_t L^6_x} \\
&\sleq \| \psi_0 \|_{\sHc^{1/2}} + \norma{ \psi }^2_{L^2_t \sWc^{-1/2,6} }  \norma{ \psi  }_{L^\infty_t \sWc^{-1/2,6} } \\
&\sleq \| \psi_0 \|_{\sHc^{1/2}} + \norma{ \psi }^2_{L^2_t \sWc^{-1/3,6} }  \norma{ \psi  }_{L^\infty_t \sHc^{1/2} } ,
\end{align*}
and one can conclude by a standard continuation argument.
\end{proof}

We also give a formulation of the Kato-Ponce inequality for the 
relativistic Sobolev spaces.

\begin{proposition} \label{katoponprop}
Let $f,g \in \cS(\R^3)$, and let $c>0$, $1 < r < \infty$ and $k \geq 0$. Then 
\begin{align} \label{katopon}
\|f \; g \|_{\sWc^{k,r}} &\sleq \|f\|_{\sWc^{k,r_1}} \|g\|_{L^{r_2}} + \|f\|_{L^{r_3}} \|g\|_{\sWc^{k,r_4}}, 
\end{align}
with \\
\begin{align*}
\frac{1}{r} = \frac{1}{r_1} + \frac{1}{r_2} = \frac{1}{r_3} + \frac{1}{r_4}, \; \; \; 1 < r_1, r_4 <+\infty.
\end{align*}
\end{proposition}

\begin{remark}
For $c=1$ Eq. \eqref{katopon} reduces to the classical Kato-Ponce inequality.
\end{remark}

\begin{proof}
We follow an argument by Cordero and Zucco (see Theorem 2.3 in \cite{cordero2011strichartz}). \\
We introduce the dilation operator $S_c(f)(x):=f(x/c)$, for any $c>0$. \\
Then we apply the classical Kato-Ponce inequality to the rescaled product 
$S_c(fg) = S_c(f) \; S_c(g)$,
\begin{align} \label{katopon1step}
\| S_c(fg) \|_{W^{k,r}} &\sleq  \|S_c(f)\|_{W^{k,r_1}} \|S_c(g)\|_{L^{r_2}} + \|S_c(f)\|_{L^{r_3}} \|S_c(g)\|_{W^{k,r_4}},
\end{align}
where
\begin{align*}
\frac{1}{r} = \frac{1}{r_1} + \frac{1}{r_2} = \frac{1}{r_3} + \frac{1}{r_4}, \; \; \; 1 < r_1, r_4 <+\infty.
\end{align*}
Now, combining the commutativity property 
\begin{align*}
\jap{\grad}^k S_c(f)(x) &= c^{-k} S_c( \jap{\grad}_c^k \; f)(x),
\end{align*}
with the equality $\|S_c(f)\|_{L^r} = c^{-3/r} \|f\|_{L^r}$, 
we can rewrite \eqref{katopon1step} as
\begin{align*}
\|\jap{\grad}^k( f \; g) \|_{L^{r}} &\sleq \|\jap{\grad}^k f\|_{L^{r_1}} \|g\|_{L^{r_2}} + \|f\|_{L^{r_3}} \|\jap{\grad}^k g\|_{L^{r_4}}, 
\end{align*}
and this leads to the thesis.
\end{proof}

\indent We conclude with another dispersive result, which could be 
interesting in itself: by exploiting the boundedness of the wave operators 
for the Schr\"odinger equation, we can deduce Strichartz estimates for 
the KG equation with potential.

\begin{theorem} \label{strpotthm}
Let $c \geq 1$, and consider the operator \\
\begin{align} \label{oppot}
\cH(x) := c (c^2-\Delta+V(x))^{1/2} &= \cH_0 (1+ \jap{\grad}_c^{-2}V)^{1/2},
\end{align}
where $V \in C(\R^3,\R)$ is a potential such that
\begin{align*}
|V(x)|+|\grad V(x)| &\sleq  \la x \ra^{-\beta}, \; \; x \in \R^3,
\end{align*}
for some $\beta>5$, and that 0 is neither an eigenvalue nor a resonance for the operator $-\Delta+V(x)$. Let $(p,q)$ be a Schr\"odinger admissible couple, and assume that $\psi_0 \in \jap{\grad}^{-1/2}_c L^2$ is orthogonal to the bound states of $-\Delta+V(x)$. Then 
\begin{align} \label{strpot}
\|\jap{\grad}_c^{ \frac{1}{q}-\frac{1}{p} } \, e^{it\cH(x)}\psi_0\|_{L^p_tL^q_x} &\sleq c^{ \frac{1}{q}-\frac{1}{p}-\frac{1}{2} } \|\jap{\grad}_c^{1/2} \; \psi_0\|_{L^2}.
\end{align}
\end{theorem}

In order to prove Theorem \ref{strpotthm} we recall Yajima's result on wave
operators \cite{yajima1995w} (where we denote by $P_c(-\Delta+V)$ the projection 
onto the continuous spectrum of the operator $-\Delta+V$).

\begin{theorem} \label{waveopthm}
Assume that 
\begin{itemize}
\item 0 is neither an eigenvalue nor a resonance for $-\Delta+V$; \\
\item $|\d^\alpha V(x)| \sleq \la x \ra^{-\beta}$ for $|\alpha|\leq k$, for some $\beta>5$.
\end{itemize}
Consider the strong limits
\begin{align*}
\cW_\pm := \lim_{t \to \pm \infty} e^{it(-\Delta+V)} e^{it\Delta}, \; &\; \cZ_\pm := \lim_{t \to \pm \infty} e^{-it\Delta} e^{it(\Delta-V)}P_c(-\Delta+V).
\end{align*}
Then $\cW_\pm: L^2 \to P_c(-\Delta+V)L^2$ are isomorphic isometries which extend into isomorphisms $\cW_\pm: W^{k,p} \to P_c(-\Delta+V)W^{k,p}$ for all $p \in [1,+\infty]$, with inverses $\cZ_\pm$. Furthermore, for any Borel function $f(\cdot)$ we have 
\begin{align} \label{intertwining}
f(-\Delta+V)P_c(-\Delta+V) = \cW_\pm f(-\Delta) \cZ_\pm, \; &\; f(-\Delta) = \cZ_\pm f(-\Delta+V)P_c(-\Delta+V)\cW_\pm.
\end{align}
\end{theorem}

Now, in the case $c=1$ one can derive Strichartz estimates for $\cH(x)$ from 
the Strichartz estimates for the free KG equation, just by applying the 
aforementioned Theorem by Yajima in the case $k=1$ (since 
$1/p-1/q+1/2 \in [0,5/6]$ for all Schr\"odinger admissible couples $(p,q)$). 
This was already proved in \cite{bambusi2011dispersion} (see Lemma 6.3). 
In the general case, this will follow from an interpolation theory argument, 
and we defer it to Appendix \ref{interp}.

\section{Galerkin Averaging Method} \label{Galavmethod} 

\indent Consider the scale of Banach spaces 
$W^{k,p}(M,\C^n \times \C^n) \ni (\psi,\bar\psi)$  
($k \geq 1$, $1<p<+\infty$, $n \in \N_0$) 
endowed by the standard  symplectic form. 
Having fixed $k$ and $p$, and $U_{k,p} \subset W^{k,p}$ open, 
we define the gradient of $H \in C^\infty(U_{k,p},\R)$ w.r.t. $\bar\psi$ 
as the unique function s.t.
\begin{align*}
\la \grad_{\bar\psi} H ,\bar h \ra &= \di_{\bar\psi}H \bar h, \; \; \forall h \in W^{k,p},
\end{align*}
so that the Hamiltonian vector field of a Hamiltonian 
function H is given by \\
\[ X_H(\psi,\bar\psi)=(i\grad_{\bar\psi}H, \; -i\grad_{\psi}H). \]
The open ball of radius $R$ and center $0$ in $W^{k,p}$ will be denoted by 
$B_{k,p}(R)$. \\

\indent Now, we call an \emph{admissible family of cut-off (pseudo-differential) operators} a sequence $(\pi_j(D))_{j \geq 0}$, where 
$\pi_j(D): W^{k,p} \to W^{k,p}$
for any $j\geq0$,  such that 
\begin{itemize}
\item for any $j\geq0$ and for any $f \in W^{k,p}$
\begin{align*}
f = \sum_{j \geq 0} \pi_j(D) f;
\end{align*}
\item for any $j\geq0$ $\pi_j(D)$ can be extended to a self-adjoint operator 
on $L^2$, and there exist constants $K_1$, $K_2>0$ such that
\begin{align*}
K_1 \left( \sum_{j\geq0} \|\pi_j(D)f\|_{L^2}^2 \right)^{1/2} &\leq 
\|f\|_{L^2} \leq K_2 \left( \sum_{j\geq0} \|\pi_j(D)f\|_{L^2}^2 \right)^{1/2};
\end{align*}
\item for any $j \geq 0$, if we denote by $\Pi_j(D) := \sum_{l=0}^j \pi_l(D)$, 
there exist positive constants $K'$, (possibly depending on $k$ and $p$) 
such that
\begin{align*}
\|\Pi_jf\|_{k,p} &\leq K' \, \|f\|_{k,p} \; \; \forall f \in W^{k,p};
\end{align*}
\item there exist positive constants $K''_1$, $K''_2$ 
(possibly depending on $k$ and $p$) and an increasing and unbounded sequence 
$(K_j)_{j \in \N} \subset \R_+$ such that
\begin{equation} \label{equivnorms}
K''_1 \| f \|_{W^{k,p}} \leq \left\| \left[ \sum_{j \in \N} K_j^{2k} |\pi_j(D)f|^2 \right]^{1/2} \right\|_{L^p} \leq K''_2 \|f\|_{W^{k,p}}.
\end{equation}
\end{itemize}

\begin{remark} \label{fourierproj}
Let $k \geq 0$, $M$ be either $\R^d$ or the d-dimensional torus $\T^d$, and 
consider the Sobolev space $H^k=H^k(M)$.
One can readily check that Fourier projection operators on $H^k$
\begin{align*}
\pi_j \psi(x) := (2\pi)^{-d/2} \int_{j-1 \leq |k| \leq j} \hat\psi(k) e^{i k \cdot x} \di k, \; \; j \geq 1
\end{align*}
form an admissible family of cut-off operators. In this case we have 
\begin{align*}
\Pi_N \psi(x) := (2\pi)^{-d/2} \int_{|k| \leq N} \hat\psi(k) e^{i k \cdot x} \di k, \; \; N \geq 0,
\end{align*}
and the constants $(K_j)_{j\in \N}$ in \eqref{equivnorms} are given by $K_j:=j$.
\end{remark}

\begin{remark} \label{littlepaley}
Let $k \geq 0$, $1 < p < +\infty$, we now introduce the 
Littlewood-Paley decomposition on the Sobolev space $W^{k,p}=W^{k,p}(\R^d)$ 
(see \cite{taylor2011partial}, Ch. 13.5). \\
\indent In order to do this, define the cutoff operators in $W^{k,p}$ 
in the following way: 
start with a smooth, radial nonnegative function $\phi_0: \R^d \to \R$
such that $\phi_0(\xi) = 1$ for $|\xi| \leq 1/2$, and 
$\phi_0(\xi) = 0$ for $|\xi| \geq 1$; 
then define $\phi_1(\xi):=\phi_0(\xi/2)-\phi_0(\xi)$, and set
\begin{align} \label{litpal}
\phi_j(\xi) &:= \phi_1(2^{1-j}\xi), \; \; j \geq 2.
\end{align}
Then $(\phi_j)_{j \geq 0}$ is a partition of unity,
\begin{align*}
\sum_{j \geq 0} \phi_j(\xi) &= 1.
\end{align*}
Now, for each $j \in \N$ and each $f \in W^{k,2}$, we can define $\phi_j(D)f$ by
\begin{align*}
\cF( \phi_j(D)f )(\xi) := \phi_j(\xi)\hat{f}(\xi).
\end{align*}
It is well known that for $p \in (1,+\infty)$ the map $\Phi:L^p(\R^d) \to L^p(\R^d,l^2)$,
\begin{align*}
\Phi(f) &:= (\phi_j(D)f)_{j \in \N},
\end{align*}
maps $L^p(\R^d)$ isomorphically onto a closed subspace of $L^p(\R^d,l^2)$, and we have compatibility of norms (\cite{taylor2011partial}, Ch. 13.5, (5.45)-(5.46)), 
\begin{align*}
K'_p \| f \|_{L^p} \leq \|\Phi(f)\|_{L^p(\R^d,l^2)} &:= \left\| \left[ \sum_{j \in \N} |\phi_j(D)f|^2 \right]^{1/2} \right\|_{L^p} \leq K_p \|f\|_{L^p},
\end{align*}
and similarly for the $W^{k,p}$-norm, i.e. for any $k>0$ and $p \in (1,+\infty)$
\begin{align} \label{compnorms}
K'_{k,p} \| f \|_{W^{k,p}} \leq \left\| \left[ \sum_{j \in \N} 2^{2jk} |\phi_j(D)f|^2 \right]^{1/2} \right\|_{L^p} \leq K_{k,p} \|f\|_{W^{k,p}}.
\end{align}
We then define the cutoff operator $\Pi_N$ by
\begin{align} \label{cutoff}
 \Pi_N\psi := \sum_{j \leq N}\phi_j(D)\psi.
\end{align}
Hence, according to the above definition, the sequence 
$(\phi_j(D))_{j\geq0}$ is an admissible family of cut-off operators. \\
We point out that the Littlewood-Paley decomposition, along with equality 
\eqref{compnorms}, can be extended to compact manifolds (see \cite{burq2004strichartz}), 
as well as to some particular non-compact manifolds (see \cite{bouclet2010littlewood}).
\end{remark}

\indent Now we consider a Hamiltonian system of the form \\
\begin{equation} \label{absH}
H=h_0+ \epsilon \, h + \epsilon \, F, 
\end{equation}
where $\epsilon>0$ is a parameter. We fix an admissible family of cut-off 
operators $(\pi_j(D))_{j \geq 0}$ on $W^{k,p}(\R^d)$. We assume that
\begin{itemize}
\item[PER]  $h_0$ generates a linear periodic flow $\Phi^t$ with period $2\pi$, 
\[ \Phi^{t+2\pi} = \Phi^t \; \; \forall t. \]
We also assume that $\Phi^t$ is analytic from $W^{k,p}$ to itself 
for any $k \geq 1$, and for any $p \in (1,+\infty)$;
\item[INV] for any $k\geq 1$, for any $p \in (1,+\infty)$,  
$\Phi^t$ leaves invariant the space $\Pi_jW^{k,p}$ for any $j\geq0$. 
Furthermore, for any $j \geq 0$ 
\[ \pi_j(D) \circ \Phi^t = \Phi^t \circ \pi_j(D); \]
\item[NF] $h$ is in normal form, namely
\[ h \circ \Phi^t = h. \]
\end{itemize}
Next we assume that both the Hamiltonian and the vector field 
of both $h$ and $F$  admit an asymptotic expansion in $\epsilon$ of the form
\begin{align} \label{Hexp}
h \sim  \sum_{j \geq 1} \epsilon^{j-1} h_j, &\; \; F \sim \sum_{j \geq 1} \epsilon^{j-1} F_j, \\
X_h \sim \sum_{j \geq 1} \epsilon^{j-1} X_{h_j}, &\; \; X_F \sim \sum_{j \geq 1} \epsilon^{j-1} X_{F_j},
\end{align}
and that the following properties are satisfied
\begin{itemize}

\item[HVF] There exists $R^\ast>0$ such that for any $j \geq 1$ 
\begin{itemize}
\item[$\cdot$] $X_{h_j}$ is analytic from $B_{k+2j,p}(R^\ast)$ to $W^{k,p}$;
\item[$\cdot$] $X_{F_j}$ is analytic from $B_{k+2(j-1),p}(R^\ast)$ to $W^{k,p}$.
\end{itemize}
Moreover, for any $r \geq 1$ we have that 
\begin{itemize}
\item[$\cdot$] $X_{h-\sum_{j=1}^r \epsilon^{j-1} h_j}$ is analytic from $B_{k+2(r+1),p}(R^\ast)$ to $W^{k,p}$;
\item[$\cdot$] $X_{F - \sum_{j=1}^r \epsilon^{j-1} F_j}$ is analytic from $B_{k+2r,p}(R^\ast)$ to $W^{k,p}$.
\end{itemize}

\end{itemize}

The main result of this section is the following theorem.

\begin{theorem} \label{normformgavthm}
Fix $r\geq1$, $R>0$, $k_1\gg 1$, $1<p<+\infty$. 
Consider \eqref{absH}, and assume PER, INV 
(with respect to the Littlewood-Paley decomposition),
NF and HVF.
Then $\exists$ $k_0=k_0(r)>0$ with the following properties: 
for any $k \geq k_1$ there exists $\epsilon_{r,k,p} \ll 1$ such that 
for any $\epsilon<\epsilon_{r,k,p}$ there exists 
$\cT^{(r)}_\epsilon:B_{k,p}(R) \to B_{k,p}(2R)$ analytic canonical transformation 
such that
\[ H_r := H \circ \cT^{(r)}_\epsilon = h_0 + \sum_{j=1}^r\epsilon^j \cZ_j + \epsilon^{r+1} \; \mathcal{R}^{(r)}, \] \\
where $\cZ_j$ are in normal form, namely
\begin{align} \label{NFthm}
\{\cZ_j,h_0\} &= 0,
\end{align} 
and
\begin{align*} 
\sup_{B_{k+k_0,p}(R)} \|X_{\cZ_{j}}\|_{W^{k,p}} &\leq C_{k,p},
\end{align*} 
\begin{align} \label{Remthm}
\sup_{B_{k+k_0,p}(R)} \|X_{\mathcal{R}^{(r)}}\|_{W^{k,p}} &\leq C_{k,p},
\end{align} 
\begin{align} \label{CTthm}
\sup_{B_{k,p}(R)} \|\cT^{(r)}_\epsilon-id\|_{W^{k,p}} &\leq C_{k,p} \, \epsilon.
\end{align}
In particular, we have that \\
\[ \cZ_1(\psi,\bar\psi) = h_1(\psi,\bar\psi) + \la F_1 \ra(\psi,\bar\psi), \\ \]
where $\la F_1 \ra(\psi,\bar\psi) := \int_0^{2\pi} F_1\circ\Phi^t(\psi,\bar\psi) \frac{\di t}{2\pi}$. \\
\end{theorem}

\section{Proof of Theorem \ref{normformgavthm} } \label{Galavproof}
\sectionmark{Proof of Theorem} 

We first make a Galerkin cutoff through the Littlewood-Paley decomposition 
(see \cite{taylor2011partial}, Ch. 13.5). \\
\indent In order to do this, fix $N \in \mathbb{N}$, $N \gg 1$, 
and introduce the cutoff operators $\Pi_N$ in $W^{k,p}$ by 
\begin{align*}
 \Pi_N\psi &:= \sum_{j \leq N}\phi_j(D)\psi,
\end{align*}
where $\phi_j(D)$ are the operators we introduced in Remark \ref{littlepaley}. 

We notice that by assumption INV the Hamiltonian vector field of $h_0$ 
generates a continuous flow $\Phi^t$ which leaves $\Pi_NW^{k,p}$ invariant. \\
Now we set $H = H_{N,r} + \cR_{N,r} + \cR_r$, where \\
\begin{align} \label{truncsys}
H_{N,r} &:= h_{0} + \epsilon \, h_{N,r} +  \epsilon \, F_{N,r}, \\ 
h_{N,r} &:= \sum_{j=1}^r \epsilon^{j-1} h_{j,N}, \; \; h_{j,N} := h_j \circ \Pi_N, \\
F_{N,r} &:= \sum_{j=1}^r \epsilon^{j-1} F_{j,N}, \; \; F_{j,N} := F_j \circ \Pi_N,
\end{align}
and
\begin{align} \label{remsys}
\cR_{N,r} &:= h_0 +\sum_{j=1}^r \epsilon^j h_j +\sum_{j=1}^r \epsilon^j F_j -H_{N,r}, \\
\cR_r &:= \epsilon \left( h - \sum_{j=1}^r \epsilon^{j-1} h_j \right) + \epsilon \left( F - \sum_{j=1}^r \epsilon^{j-1} F_j \right).
\end{align}

The system described by the Hamiltonian \eqref{truncsys} is the one that 
we will put in normal form.  \\
\indent In the following we will use the notation $a \sleq b$ to mean: 
there exists a positive constant $K$ independent of $N$ and $R$ 
(but dependent on $r$, $k$ and $p$), such that $a \leq Kb$. \\
\indent We exploit the following intermediate results: \\

\begin{lemma} \label{truncest}
For any $k \geq k_1$ and $p \in (1,+\infty)$ there 
exists $B_{k,p}(R) \subset W^{k,p}$ s.t. $\forall$ $\sigma >0$, $N>0$
\begin{align} \label{truncremt}
\sup_{ B_{k+\sigma+2(r+1),p}(R) } \|X_{\cR_{N,r}}(\psi,\bar\psi)\|_{W^{k,p}} &\sleq \; \frac{\epsilon}{2^{\sigma(N+1)}}, 
\end{align}
\begin{align} \label{expremest} 
\sup_{ B_{k+2(r+1),p}(R) } \|X_{\cR_r}(\psi,\bar\psi)\|_{W^{k,p}} &\sleq \epsilon^{r+1}.
\end{align}
\end{lemma}

\begin{proof}
We recall that $\cR_{N,r} = h_0 +\sum_{j=1}^r \epsilon^j h_j +\sum_{j=1}^r \epsilon^j F_j -H_{N,r}$. \\
Now, $\|id-\Pi_N\|_{ W^{k+\sigma,p} \to W^{k,p} } \sleq 2^{-\sigma(N+1)}$, since 
\begin{align*}
\left\| \sum_{j \geq N+1} \phi_j(D)f \right\|_{W^{k,p}} &\sleq \left\| \left[ \sum_{j \geq N+1} |2^{jk} \phi_j(D)f|^2 \right]^{1/2} \right\|_{L^p} \\
&\sleq 2^{-\sigma(N+1)} \left\| \left[ \sum_{j \geq N+1} |2^{j(k+\sigma)} \phi_j(D)f|^2 \right]^{1/2} \right\|_{L^p} \\
&\sleq 2^{-\sigma(N+1)} \|f\|_{W^{k+\sigma,p}},
\end{align*}
hence \\
\begin{align*}
&\sup_{\psi \in B_{k+2(r+1)+\sigma,p}(R)} \; \| X_{\cR_{N,r}}(\psi,\bar\psi)\|_{W^{k,p}} \\
&\sleq \; \|dX_{ \sum_{j=1}^r \epsilon^j(h_j+F_j) }\|_{ L^\infty(B_{k+2(r+1),p}(R),W^{k,p}) } \|id-\Pi_N\|_{ L^\infty(B_{k+2(r+1)+\sigma,p}(R),B_{k+2(r+1),p}) } \\
&\sleq \epsilon \, 2^{-\sigma(N+1)}.
\end{align*}
\indent The estimate of $X_{\cR_r}$ follow from the hypothesis HVF. \\
\end{proof}

\begin{lemma} \label{pertestlemma}
Let $j \geq 1$. 
Then for any $k \geq k_1+2(j-1)$ and $p \in (1,+\infty)$ there 
exists $B_{k,p}(R) \subset W^{k,p}$ such that
\begin{align*}
\sup_{ B_{k,p}(R) } \|X_{h_{j,N}}(\psi,\bar\psi)\|_{k,p} &\leq K^{(h)}_{j,k,p} 2^{2jN} , \\
\sup_{ B_{k,p}(R) } \|X_{F_{j,N}}(\psi,\bar\psi)\|_{k,p} &\leq K^{(F)}_{j,k,p} 2^{2(j-1)N} , 
\end{align*}
where 
\begin{align*}
K^{(h)}_{j,k,p} &:= \sup_{B_{k,p}(R) } \|X_{h_j}(\psi,\bar\psi)\|_{k-2j,p}, \\
K^{(F)}_{j,k,p} &:= \sup_{B_{k,p}(R) } \|X_{F_j}(\psi,\bar\psi)\|_{k-2(j-1),p}.
\end{align*}
\end{lemma}

\begin{proof}
It follows from \\
\begin{align}
\sup_{\psi \in B_{k,p}(R)} &\left\| \sum_{h \leq N} \phi_h(D)X_{F_{j,N}} (\psi,\bar\psi) \right\|_{W^{k,p}} \sleq 
\sup_{\psi \in B_{k,p}(R)} \left\| \left[ \sum_{h \leq N} |2^{hk} \phi_h(D)X_{F_{j,N}}(\psi,\bar\psi)|^2 \right]^{1/2} \right\|_{L^p} \\
&\leq 2^{2(j-1)N} \sup_{\psi \in B_{k,p}(R)} \left\| \left[ \sum_{h \leq N} |2^{h[k-2(j-1)]} \phi_h(D)X_{F_{j,N}}(\psi,\bar\psi)|^2 \right]^{1/2} \right\|_{L^p} \\
&\sleq 2^{2(j-1)N} \sup_{\psi \in B_{k,p}(R)} \|X_{F_{j,N}}(\psi,\bar\psi)\|_{k-2(j-1),p} \\
&= K^{(F)}_{j,k,p} \, 2^{2(j-1)N},
\end{align}
and similarly for $X_{h_{j,N}}$.
\end{proof}

Next we have to normalize the system \eqref{truncsys}. In order to do this 
we need a slight reformulation of Theorem 4.4 in \cite{bambusi1999nekhoroshev}. Here we report 
a statement of the result adapted to our context. 

\begin{lemma} \label{NFest}
Let $k \geq k_1+2r$, $p \in (1,+\infty)$, $R>0$, and consider the 
system \eqref{truncsys}. 
Assume that $\epsilon < 2^{-4Nr}$, and that 
\begin{align} \label{smallcond}
( K^{(F,r)}_{k,p} + K^{(h,r)}_{k,p} ) r 2^{2Nr} \epsilon &< 2^{-9} e^{-1} \pi^{-1} R ,
\end{align}
where 
\begin{align*}
K^{(F,r)}_{k,p} &:= \sup_{1\leq j\leq r} \sup_{\psi \in B_{k,p}(R)} \|X_{F_j}(\psi,\bar\psi)\|_{k-2(j-1),p}, \\
K^{(h,r)}_{k,p} &:= \sup_{1\leq j\leq r} \sup_{\psi \in B_{k,p}(R)} \|X_{h_j}(\psi,\bar\psi)\|_{k-2j,p}.
\end{align*}
Then there exists an analytic canonical transformation 
$\cT^{(r)}_{\epsilon,N}:B_{k,p}(R) \to B_{k,p}(2R)$ such that
\begin{align*}
\sup_{B_{k,p}(R/2)} \|\cT^{(r)}_{\epsilon,N}(\psi,\bar\psi)-(\psi,\bar\psi)\|_{W^{k,p}} &\leq 4\pi r K^{(F,r)}_{k,p} 2^{2Nr} \epsilon,
\end{align*}
and that puts \eqref{truncsys} in normal form up to a small remainder, 
\begin{align} \label{stepr}
H_{N,r} \circ \cT^{(r)}_{\epsilon,N} &= h_{0} + \epsilon h_{N,r} + 
\epsilon Z^{(r)}_N + \epsilon^{r+1} \cR^{(r)}_N, 
\end{align}
with $Z^{(r)}_N$ is in normal form, namely $\{h_{0,N},Z^{(r)}_N\}=0$, and
\begin{align}
\sup_{B_{k,p}(R/2)} \|X_{ Z^{(r)}_N }(\psi,\bar\psi)\|_{k,p} &\leq 
4 \, 2^{2Nr} \, \epsilon \, 
\left( r K^{(F,r)}_{k,p} + r K^{(h,r)}_{k,p} \right) \, r 2^{2Nr} K^{(F,r)}_{k,p} \nonumber \\
&= 4r^2 K^{(F,r)}_{k,p} ( K^{(F,r)}_{k,p} + K^{(h,r)}_{k,p} ) 2^{4Nr} \epsilon,
\label{itervf1}
\end{align} 
\begin{align} \label{vecfrem}
&\sup_{B_{k,p}(R/2)} \|X_{ \cR^{(r)}_N }(\psi,\bar\psi)\|_{k,p} \\
&\leq 2^8 e \frac{T}{R} (K^{(F,r)}_{k,p} + K^{(F,r)}_{k,p}) r 2^{2Nr} \\
& \; \times \left[ \frac{4T}{R} 
\left( 
2^9 3^2 e \frac{T}{R} (K^{(F,r)}_{k,p} + K^{(F,r)}_{k,p}) K^{(F,r)}_{k,p} r^2 2^{4Nr} \epsilon
 + 5 K^{(h,r)}_{k,p} \, r 2^{2Nr} + 5 K^{(F,r)}_{k,p} \, r 2^{2Nr} 
\right) r \right]^r
\end{align}
\end{lemma}

The proof of Lemma \ref{NFest} is postponed to Appendix \ref{BNFest}.

\begin{remark}
In the original notation of Theorem 4.4 in \cite{bambusi1999nekhoroshev} we set 
\begin{align*}
\cP &= W^{k,p}, \\
h_\omega &= h_{0}, \\
\hat h &= \epsilon h_{N,r}, \\ 
f &= \epsilon F_{N,r}, \\
f_1 &= r = g \equiv 0, \\
F &= K^{(F,r)}_{k,p} \, r 2^{2Nr} \, \epsilon, \\
F_0 &= K^{(h,r)}_{k,p} \, r 2^{2Nr}  \, \epsilon.
\end{align*}
\end{remark}

\begin{remark}
Actually, Lemma \ref{NFest} would also hold under a weaker smallness assumption 
on $\epsilon$: it would be enough that $\epsilon < 2^{-2N}$, and that
\begin{align} \label{smallcondw}
\epsilon &\, 
\left[ K^{(F,r)}_{k,p} \frac{ 1-2^{2Nr}\epsilon^r }{1-2^{2N}\epsilon} + 
K^{(h,r)}_{k,p} \frac{ 2^{2N}(1-2^{2Nr}\epsilon^r) }{1-2^{2N}\epsilon} \right]
< 2^{-9} e^{-1} \pi^{-1} R 
\end{align}
is satified. However, condition \eqref{smallcondw} is less explicit than 
\eqref{smallcond}, that allows us to apply directly the scheme of \cite{bambusi1999nekhoroshev}.
The disadvantage of the stronger smallness assumption \eqref{smallcond} is that 
it holds for a smaller range of $\epsilon$, and that at the end of the proof 
it will force us to choose a larger parameter $\sigma = 4r^2$. 
By using \eqref{smallcondw} and by making a more careful analysis, it may be 
possible to prove Theorem \ref{normformgavthm} also by choosing $\sigma = 2r$.
\end{remark}

Now we conclude with the proof of the Theorem \ref{normformgavthm}. \\

\begin{proof}
Now consider the transformation $\cT^{(r)}_{\epsilon,N}$ defined by 
Lemma \ref{NFest}, then 
\begin{align*}
(\cT^{(r)}_{\epsilon,N})^\ast H  &= h_{0} \; + \sum_{j=1}^{r} \epsilon^j h_{j,N} \;
+ \epsilon Z^{(r)}_N + \epsilon^{r+1} \cR^{(r)}_N + \epsilon^r \cR_{Gal}
\end{align*}
where we recall that
\begin{align*}
\epsilon^{r} \cR_{Gal} &:= (\cT^{(r)}_{\epsilon,N})^\ast( \cR_{N,r} + \cR_r).
\end{align*}
\indent By exploiting the Lemma \ref{NFest} we can estimate the vector field 
of $\cR^{(r)}_N$, while by using Lemma \ref{truncest} and \eqref{vfest} we get
\begin{align} \label{galremest} 
\sup_{B_{k+\sigma+2(r+1),p }(R/2)} \; \|X_{\cR_{Gal}}(\psi,\bar\psi)\|_{W^{k,p}} &\sleq 
\left( \frac{\epsilon}{2^{\sigma(N+1)}} + \frac{\epsilon^{r+1}}{\sigma+2(r+1)} \right). 
\end{align}
\indent To get the result choose \\
\begin{align*}
k_0 &= \sigma+2(r+1), \\
N &= r \sigma^{-1} \log_2(1/\epsilon)-1, \\
\sigma &= 4r^2.
\end{align*}

\end{proof}

\begin{remark}
The compatibility condition $N \geq 1$ and \eqref{smallcond} lead to
\begin{align*}
\epsilon \leq \left[ 2^{-9} e^{-1} \pi^{-1} R ( K^{(F,r)}_{k,p} + K^{(h,r)}_{k,p} )^{-1} r^{-1} 2^{-2r} \right]^{\frac{\sigma}{2r}} &=: \epsilon_{r,k,p} 
\leq 2^{-2\sigma/r} \leq 2^{-8r}.
\end{align*}
\end{remark}

\begin{remark}
We point out the fact that Theorem \ref{normformgavthm} holds for 
the scale of Banach spaces $W^{k,p}(M,\C^n \times \C^n)$, where 
$k \geq 1$, $1 < p < +\infty$, $n \in \N_0$, and where $M$ is a smooth 
manifold on which the Littlewood-Paley decomposition can be constructed, 
for example a compact manifold (see sect. 2.1 in \cite{burq2004strichartz}), 
$\R^d$, or a noncompact manifold satisfying some technical assumptions 
(see \cite{bouclet2010littlewood}). \\
\end{remark}

If we restrict to the case $p=2$, and 
we consider $M$ as either $\R^d$ or the $d$-dimensional torus $\T^d$,
we can prove an analogous result for Hamiltonians 
$H(\psi,\bar\psi)$ with $(\psi,\bar\psi) \in H^k:=W^{k,2}(M,\C \times \C)$. 
In the following we denote by $B_k(R)$ 
the open ball of radius $R$ and center $0$ in $H^k$.
We recall that the Fourier projection operator on $H^k$ is given by
\begin{align*}
\pi_j \psi(x) := (2\pi)^{-d/2} \int_{j-1 \leq |k| \leq j} \hat\psi(k) e^{i k \cdot x} \di k, \; \; j \geq 1.
\end{align*}

\begin{theorem} \label{normformgavthm2}
Fix $r \geq 1$, $R>0$, $k_1 \gg 1$. Consider \eqref{absH}, and assume PER, 
INV (with respect to Fourier projection operators), NF and HVF. 
Then $\exists$ $k_0=k_0(r)>0$ with the following properties: 
for any $k \geq k_1$ there exists $\epsilon_{r,k} \ll 1$ such that 
for any $\epsilon<\epsilon_{r,k}$ there exists 
$\cT^{(r)}_\epsilon:B_k(R) \to B_k(2R)$ transformation s.t. \\
\[ H_r := H \circ \cT^{(r)}_\epsilon = h_0 + \sum_{j=1}^r\epsilon^j \cZ_j + \epsilon^{r+1} \; \mathcal{R}^{(r)}, \] \\
where $\cZ_j$ are in normal form, namely
\begin{align} \label{NFthm2}
\{\cZ_j,h_0\} &= 0,
\end{align} 
and
\begin{align} \label{Remthm2}
\sup_{B_{k+k_0}(R)} \|X_{\mathcal{R}^{(r)}}\|_{H^k} &\leq C_k,
\end{align} 
\begin{align} \label{CTthm2}
\sup_{B_k(R)} \|\cT^{(r)}_\epsilon-id\|_{H^k} &\leq C_k \, \epsilon.
\end{align}
In particular, we have that \\
\[ \cZ_1(\psi,\bar\psi) = h_1(\psi,\bar\psi) + \la F_1 \ra(\psi,\bar\psi), \\ \]
where $\la F_1 \ra(\psi,\bar\psi) := \int_0^{2\pi} F_1\circ\Phi^t(\psi,\bar\psi) \frac{\di t}{2\pi}$. \\
\end{theorem}

The only technical difference between the proofs of 
Theorem \ref{normformgavthm} and the proof of Theorem \ref{normformgavthm2} 
is that we exploit the Fourier cut-off operator
\begin{align*}
\Pi_N \psi(x) := \int_{|k| \leq N} \hat\psi(k) e^{i k \cdot x} \di k,
\end{align*}
as in \cite{bambusi2005galerkin}. This in turn affects \eqref{truncremt}, 
which in this case reads 
\begin{align} \label{truncremt2}
\sup_{ B_{k+\sigma+2(r+1)}(R) } \|X_{\cR_{N,r}}(\psi,\bar\psi)\|_{H^k} &\sleq \; \frac{\epsilon}{N^\sigma}, 
\end{align}
and \eqref{galremest}, for which we have to choose a bigger cut-off, 
$N=\epsilon^{- r \sigma}$. 

\section{Application to the nonlinear Klein-Gordon equation} \label{NLKGappl}

\subsection{The real nonlinear Klein-Gordon equation}

We first consider the Hamiltonian of the real non-linear Klein-Gordon equation
with power-type nonlinearity on a smooth manifold $M$ 
($M$ is such the Littlewood-Paley decomposition is well-defined; take, for example, a smooth compact manifold, or $\R^d$). The Hamiltonian is of the form
\begin{align} \label{NLKGham}
 H(u,v) &= \frac{c^2}{2} \la v,v\ra + \frac{1}{2} \la u,\jap{\grad}_c^2u \ra \; + \; \lambda \int \frac{u^{2l}}{2l},
\end{align}
where $\jap{\grad}_c:=(c^2-\Delta)^{1/2}$, $\lambda \in \mathbb{R}$, $l \geq 2$. \\
If we introduce the complex-valued variable 
\begin{align}
 \psi &:= \frac{1}{\sqrt{2}} \left[ \left(\frac{\jap{\grad}_c}{c} \right)^{1/2} u - i \left(\frac{c}{\jap{\grad}_c}\right)^{1/2}v \right], \label{changevar}
\end{align}
(the corresponding symplectic 2-form becomes $i \di\psi \wedge \di\bar\psi$), 
the Hamiltonian \eqref{NLKGham} in the coordinates $(\psi,\bar\psi)$ is
\begin{align} \label{NLKGhamnew}
H(\bar\psi,\psi) &= \la \bar{\psi}, c\jap{\grad}_c\psi \ra 
+ \frac{\lambda}{2l} \int \left[ 
\left( \frac{c}{\jap{\grad}_c} \right)^{1/2} \frac{\psi+\bar\psi}{\sqrt{2}} 
\right]^{2l} \di x.
\end{align}
If we rescale the time by a factor $c^{2}$, the Hamiltonian takes the form 
\eqref{absH}, with $\epsilon = \frac{1}{c^2}$, and 
\begin{align} \label{NLKG}
H(\psi,\bar\psi) &= h_0(\psi,\bar\psi) + \epsilon \, h(\psi,\bar\psi) + \epsilon \, F(\psi,\bar\psi),
\end{align}
where

\begin{align}
h_0(\psi,\bar\psi) &= \la \bar\psi,\psi \ra, \\
h(\psi,\bar\psi) &= \la \bar\psi, \left( c \jap{\grad}_c - c^2 \right)\psi \ra \sim \sum_{j\geq 1}\epsilon^{j-1} \; \la\bar\psi,a_j\Delta^j\psi\ra =: \sum_{j\geq 1}\epsilon^{j-1} h_j(\psi,\bar\psi),  \label{hath} \\
F(\psi,\bar\psi) &= \frac{\lambda}{2^{l+1}l} \int \left[ \left(\frac{c}{\jap{\grad}_c}\right)^{1/2} (\psi+\bar\psi) \right]^{2l} \di x \\
&\sim \frac{\lambda}{2^{l+1}l} \int (\psi+\bar\psi)^{2l} \di x \nonumber \\
&- \epsilon b_2 \int \left[ (\psi+\bar\psi)^{2l-1}\Delta(\psi+\bar\psi) 
+ \ldots 
+ (\psi+\bar\psi)\Delta((\psi+\bar\psi)^{2l-1}) \right] \di x \nonumber \\
&+ \cO(\epsilon^2) \nonumber \\
&=: \sum_{j \geq 1} \epsilon^{j-1} \, F_j(\psi,\bar\psi), \label{HP}
\end{align}
where $(a_j)_{j \geq 1}$ and $(b_j)_{j \geq 1}$ are real coefficients,
and $F_j(\psi,\bar\psi)$ is a polynomial function of the variables 
$\psi$ and $\bar\psi$ (along with their derivatives) and which admits a
bounded vector field from a neighborhood of the origin in $W^{k+2(j-1),p}$ 
to $W^{k,p}$ for any $1<p<+\infty$. 

This description clearly fits the scheme treated in the previous section,
and one can easily check that assumptions PER, NF and HVF are satisfied. 
Therefore we can apply Theorem \ref{normformgavthm} to the 
Hamiltonian \eqref{NLKG}. \\

\begin{remark} \label{1steprem}
\indent About the normal forms obtained by applying Theorem 
\ref{normformgavthm}, we remark that in the first step (case $r=1$ in 
the statement of the Theorem) the homological equation we get is of the form 
\begin{equation} \label{homeq1step}
\{\chi_1,  h_0 \} + F_1 = \la F_1 \ra, 
\end{equation}
where $F_1(\psi,\bar\psi) = \frac{\lambda}{2^{l+1}l} \int (\psi+\bar\psi)^{2l} \di x$. Hence the transformed Hamiltonian is of the form 
\begin{equation} \label{ham1step}
H_1(\psi,\bar\psi) = h_0(\psi,\bar\psi) + \frac{1}{c^2} \left[ -\frac{1}{2} \la\bar\psi,\Delta\psi\ra + \la F_1 \ra(\psi,\bar\psi) \right] + \frac{1}{c^4} \cR^{(1)}(\psi,\bar\psi).
\end{equation}
If we neglect the remainder and we derive the corresponding 
equation of motion for the system, we get 
\begin{equation} \label{eqstep1}
 -i \psi_t \; = \psi + \frac{1}{c^2} \left[ -\frac{1}{2} \Delta\psi + \frac{\lambda}{2^{l+1}} \binom{2l}{l} |\psi|^{2(l-1)}\psi \right], \\
\end{equation}
which is the NLS, and the Hamiltonian which generates the canonical 
transformation is given by 
\begin{equation} \label{chi1}
\chi_1(\psi,\bar\psi) = \frac{\lambda}{2^{l+1}l} \sum_{\substack{j=0,\ldots,2l \\ j \neq l}} \frac{1}{i \, 2(l-j)} \binom{2l}{j} \int \psi^{2l-j} \bar\psi^{j} \di x. \\
\end{equation}
\end{remark}

\begin{remark} \label{2steprem}
Now we iterate the construction by passing to the case $r=2$, 
and for simplicity we consider only the case $l=2$, which at the first step 
yields the cubic NLS. In this case one has that
\begin{align*}
\chi_1(\psi,\bar\psi) &= \int_0^T \tau \; [F_1(\Phi^\tau(\psi,\bar\psi)) \; - \; \la F_1 \ra(\Phi^\tau(\psi,\bar\psi))] \; \frac{\di\tau}{T} \\
&=  \frac{\lambda}{16} \int_0^{2\pi} \tau \int \left[ |e^{i\tau}\psi+e^{-i\tau}\bar{\psi}|^4 -6 |\psi|^4 \right] \; \di x \frac{\di \tau}{2\pi}. \\
\end{align*}
Since \\
\[ |e^{i\tau}\psi+e^{-i\tau}\bar{\psi}|^4= e^{4i\tau}\psi^4+4e^{2i\tau}\psi^3\bar{\psi}+6\psi^2\bar{\psi}^2+4e^{-2i\tau}\psi\bar{\psi}^3+e^{-4i\tau}\bar{\psi}^4 \]
and since $\int_0^{2\pi}\tau e^{in\tau} \di \tau = \frac{2\pi}{i \; n}$ for any 
non-zero integer $n$, we finally get 
\[ \chi_1(\psi,\bar\psi) = \frac{\lambda}{16} \int \frac{\psi^4-\bar{\psi}^4}{4i} + \frac{2}{i} (\psi^3\bar{\psi}-\psi\bar{\psi}^3) \; \di x. \]

If we neglect the remainder of order $c^{-6}$, we have that 
\begin{align} 
H \circ \cT^{(1)} &= h_0 + \frac{1}{c^2} h_1 + \frac{1}{c^4} \{\chi_1, h_1\} + \frac{1}{c^4} h_2 + \nonumber \\
&+ \frac{1}{c^2} \la F_1 \ra + \frac{1}{c^4} \{\chi_1, F_1\} + \frac{1}{2c^4} \{\chi_1,\{\chi_1,h_0\}\} + \frac{1}{c^4} F_2 \\
&= h_0 + \frac{1}{c^2} \left[ h_1 + \la F_1\ra \right] + \frac{1}{c^4} \left[ \{\chi_1,h_1\} + h_2 + \{\chi_1,F_1\} + \frac{1}{2} \{ \chi_1, \la F_1 \ra - F_1 \} + F_2 \right],
\end{align}
where $h_1(\psi,\bar\psi) = -\frac{1}{2} \la\bar\psi,\Delta\psi\ra$. \\
\indent Now we compute the terms of order $\frac{1}{c^4}$. 
\begin{align}
  \{\chi_1, h_1 \} &= \di \chi_1 X_{h_1} =  \frac{\d \chi_1}{\d\psi} \cdot i \frac{\d h_1}{\d\bar\psi} - i \frac{\d \chi_1}{\bar\psi} \frac{\d h_1}{\d\psi} \\
&= - \frac{\lambda}{32} \int \left[ \Delta\psi \left( \psi^3 + 6 \psi^2\bar\psi -2 \bar\psi^3 \right) - \Delta\bar\psi (2 \psi^3-6  \psi \bar\psi^2 - \bar\psi^3) \right],
\end{align}
\begin{align}
h_2 = -\frac{1}{8} \la \bar\psi,\Delta^2\psi \ra, 
\end{align}
\begin{align}
\{\chi_1, F_1\} &= \frac{\lambda^2}{32} \int 
(4\psi^3 + 12\psi^2\bar\psi +12 \psi \bar\psi^2 +4\bar\psi^3)(\psi^3+6 \psi^2\bar\psi - 2 \bar\psi^3) + \\
&- (4\psi^3 + 12 \psi^2\bar\psi + 12 \psi \bar\psi^2 + 4 \bar\psi^3)(2\psi^3 - 6 \psi\bar\psi^2 - \bar\psi^3) \; \di x,
\end{align}
\begin{align}
\{ \chi_1, \la F_1 \ra \} = \frac{\lambda^2}{2} \int \left[ |\psi|^2\psi \; (\psi^3 + 6 \psi^2\bar\psi - 2\bar\psi^3) - |\psi|^2\bar\psi \; (2 \psi^3-6\psi\bar\psi^2- \bar\psi^3)  \right]\; \di x,
\end{align}
\begin{align}
F_2 = \frac{\lambda}{16} \int \left[(\psi^3+3\psi^2 \bar\psi+3\psi \bar\psi^2+ \psi^3) \, \Delta\psi + (\bar\psi^3+3\bar\psi^2 \psi + 3 \bar\psi \psi^2 + \psi^3) \, \Delta\bar\psi \right]\; \di x.
\end{align}

Now, one can easily verify that 
$\la \{\chi_1, h_1 \}\ra = \la \{ \chi_1, \la F_1 \ra \} \ra = 0$, and that \\
\begin{align}
\la \{\chi_1, F_1 \} \ra &= \frac{\lambda^2}{32} \int 
(-8|\psi|^6+72|\psi|^6+4|\psi|^6)+ (4|\psi|^6 + 72 |\psi|^6 - 8 |\psi|^6) \; \di x \\
&= \frac{17}{4} \lambda^2 \int |\psi|^6 \; \di x,
\end{align}
\begin{align}
\la F_2 \ra &= \frac{\lambda}{16} \int 3\psi \bar\psi^2 \, \Delta\psi + 3 \bar\psi \psi^2 \, \Delta\bar\psi \; \di x \\
&= \frac{\lambda}{16} \int 3|\psi|^2(\psi \, \Delta\psi + \psi \Delta\bar\psi) \; \di x. 
\end{align}

Hence, up to a remainder of order $O\left(\frac{1}{c^6}\right)$, we have that
\begin{align} \label{ham2step}
H_2 &= h_0 + \frac{1}{c^2} \int \left[ -\frac{1}{2}  \la \bar\psi,\Delta\psi \ra + \frac{3}{8} \lambda |\psi|^4 \right] \; \di x  \nonumber\\
&+  \frac{1}{c^4} \int \left[ \frac{17}{8} \lambda^2 |\psi|^6 + \frac{3}{16} \lambda |\psi|^2(\bar\psi \, \Delta\psi + \psi \, \Delta\bar\psi) - \frac{1}{8} \la \bar\psi, \Delta^2\psi \ra \right] \; \di x,
\end{align}

which, by neglecting $h_0$ (that yields only a gauge factor) and 
by rescaling the time, leads to the following equations of motion \\
\begin{align} \label{eqstep2}
-i \psi_t \; &= \;  - \frac{1}{2} \Delta\psi + \frac{3}{4} \lambda |\psi|^2\psi \nonumber \\
&+ \frac{1}{c^2} \left[ \frac{51}{8} \lambda^2 |\psi|^4\psi + \frac{3}{16} \lambda \left(2|\psi|^2 \, \Delta\psi + \psi^2 \Delta\bar\psi + \Delta(|\psi|^2\bar\psi) \right) - \frac{1}{8} \Delta^2\psi \right].
\end{align}

To the author's knowledge, Eq. \eqref{eqstep2} has never been studied before. 
It is the nonlinear analogue of a linear higher-order Schr\"odinger equation 
that appears in \cite{carles2012higher} and \cite{carles2015higher} in the context of 
semi-relativistic equations. Indeed, the linearization of Eq. \eqref{eqstep2} 
is studied within the framework of relativistic quantum field theory, 
as an approximation of nonlocal kinetic terms; 
Carles, Lucha and Moulay studied the well-posedness of these approximations, 
as well as the convergence of the equations as the order of truncation goes to 
infinity, in the linear case, also when one takes into account the effects of 
some time-independent potentials 
(e.g. bounded potentials, the harmonic-oscillator potential and the Coulomb potential). 

Apparently, little is known for the nonlinear equation \eqref{eqstep2}: 
we just mention \cite{cui2007well}, in which the well-posedness 
of a higher-order Schrodinger equation has been studied, and 
\cite{pausader2013scattering}, in which the scattering theory for a fourth-order 
Schr\"odinger equation in dimensions $1 \leq d \leq 4$ is studied.
\end{remark}

\subsection{The complex nonlinear Klein-Gordon equation}

Now we consider the Hamiltonian of the complex non-linear Klein-Gordon equation
with power-type nonlinearity on a smooth manifold $M$ (take, for example, a smooth compact manifold, or $\R^d$) 
\begin{align} \label{cNLKGham}
 H(w,p_w) &= \frac{c^2}{2} \la p_w,p_w \ra + \frac{1}{2} \la w,\jap{\grad}_c^2w \ra \; + \; \lambda \int \frac{|w|^{2l}}{2l},
\end{align}
where $w:\R \times M \to \C$, $\jap{\grad}_c:=(c^2-\Delta)^{1/2}$, 
$\lambda \in \R$, $l \geq 2$. \\
If we rewrite the Hamiltonian in terms of $u:= Re(w)$ and $v:= Im(w)$, we have 
\begin{align} \label{cNLKGham2}
 H(u,v,p_u,p_v) &= \frac{c^2}{2} (\la p_u,p_u \ra + \la p_v,p_v \ra) 
+ \frac{1}{2} ( |\grad u|^2 + |\grad v|^2 ) + \frac{c^2}{2} ( u^2 + v^2 ) 
+ \lambda \int \frac{(u^2+v^2)^{l}}{2l}.
\end{align}
We will consider by simplicity only the cubic case ($l=2$), but the 
argument may be readily generalized to the other power-type nonlinearities.

If we introduce the variables 
\begin{align}
\psi &:= \frac{1}{\sqrt{2}} \left[ \left(\frac{\jap{\grad}_c}{c} \right)^{1/2} u - i \left(\frac{c}{\jap{\grad}_c}\right)^{1/2}p_u \right], \\
\phi &:= \frac{1}{\sqrt{2}} \left[ \left(\frac{\jap{\grad}_c}{c} \right)^{1/2} v + i \left(\frac{c}{\jap{\grad}_c}\right)^{1/2}p_v \right], \label{cchangevar}
\end{align}
(the corresponding symplectic 2-form becomes 
$i \di\psi \wedge \di\bar\psi -i \di\phi \wedge \di\bar\phi$), 
the Hamiltonian \eqref{cNLKGham} in the coordinates 
$(\psi,\phi,\bar\psi,\bar\phi)$ reads
\begin{align} \label{cNLKGhamnew}
H(\psi,\phi,\bar\psi,\bar\phi) &= \la \bar\psi, c\jap{\grad}_c\psi \ra 
+ \la \bar\phi, c\jap{\grad}_c\phi \ra \\
&+ \frac{\lambda}{16} \int_M \left[ \la \psi+\bar\psi, \frac{c}{\jap{\grad}_c} (\psi+\bar\psi) \ra + \la \phi+\bar\phi, \frac{c}{\jap{\grad}_c} (\phi+\bar\phi)  \ra \right]^{2} \di x,
\end{align}
with corresponding equations of motion \\
\begin{equation*}
\begin{cases}
-i \psi_t &= c\jap{\grad}_c\psi + \frac{1}{4} \left[ \la \psi+\bar\psi, \frac{c}{\jap{\grad}_c} (\psi+\bar\psi) \ra + \la \phi+\bar\phi, \frac{c}{\jap{\grad}_c} (\phi+\bar\phi) \ra \right] 
\frac{c}{\jap{\grad}_c} (\psi+\bar\psi), \\ \\
i \phi_t &= c\jap{\grad}_c\phi + \frac{1}{4} \left[ \la \psi+\bar\psi, \frac{c}{\jap{\grad}_c} (\psi+\bar\psi) \ra + \la \phi+\bar\phi, \frac{c}{\jap{\grad}_c} (\phi+\bar\phi) \ra \right] 
\frac{c}{\jap{\grad}_c} (\phi+\bar\phi). \\
\end{cases}
\end{equation*}
If we rescale the time by a factor $c^{2}$, the Hamiltonian takes the form 
\eqref{absH}, with $\epsilon = \frac{1}{c^2}$, and 
\begin{align} \label{cNLKG}
H(\psi,\phi,\bar\psi,\bar\phi) &= H_0(\psi,\phi,\bar\psi,\bar\phi) 
+ \epsilon \, h(\psi,\phi,\bar\psi,\bar\phi) 
+ \epsilon \, F(\psi,\phi,\bar\psi,\bar\phi),
\end{align}
where
\begin{align}
H_0(\psi,\phi,\bar\psi,\bar\phi) &= \la \bar\psi,\psi \ra + \la \bar\phi,\phi \ra, \\
h(\psi,\phi,\bar\psi,\bar\phi) &= \la \bar\psi, \left( c \jap{\grad}_c - c^2 \right)\psi \ra - \la \bar\phi, \left( c \jap{\grad}_c - c^2 \right)\phi \ra \nonumber \\
&\sim \sum_{j\geq 1}\epsilon^{j-1} \; ( \la\bar\psi,a_j\Delta^j\psi\ra + \la\bar\phi,a_j\Delta^j\phi\ra ) \nonumber \\
&=: \sum_{j\geq 1}\epsilon^{j-1} ( h_j(\psi,\phi,\bar\psi,\bar\phi) ),  \label{hatH} \\
F(\psi,\phi,\bar\psi,\bar\phi) &= \frac{\lambda}{16} \int_\T \left[ \la \psi+\bar\psi, \frac{c}{\jap{\grad}_c} (\psi+\bar\psi) \ra + \la \phi+\bar\phi, \frac{c}{\jap{\grad}_c} (\phi+\bar\phi) \ra \right]^{2} \di x, \nonumber \\
&\sim \frac{\lambda}{16} \int \left[ |\psi+\bar\psi|^2 + |\phi+\bar\phi|^2 \right]^2 \di x \nonumber \\
&+ \cO(\epsilon) \nonumber \\
&=: \sum_{j \geq 1} \epsilon^{j-1} \,F_j(\psi,\phi,\bar\psi,\bar\phi), \label{HPer}
\end{align}
where $(a_j)_{j \geq 1}$ are real coefficients, and 
$F_j(\psi,\phi,\bar\psi,\bar\phi)$ is a polynomial function of the 
variables $\psi$, $\phi$, $\bar\psi$, $\bar\phi$ (along with their derivatives) 
and which admits a bounded vector field from a neighborhood of the origin in 
$W^{k+2(j-1),p}(\R^d,\C^2 \times \C^2)$ to $W^{k,p}(\R^d,\C^2 \times \C^2)$ 
for any $1<p<+\infty$. 

This description clearly fits the scheme treated in sect. \ref{Galavmethod} with
 $n=2$, and one can easily check that assumptions PER, NF and HVF are satisfied. Therefore we can apply Theorem \ref{normformgavthm} to the 
Hamiltonian \eqref{cNLKG}. \\

\begin{remark} \label{1stepcrem}
\indent About the normal forms obtained by applying Theorem 
\ref{normformgavthm}, we remark that in the first step (case $r=1$ in 
the statement of the Theorem) the homological equation we get is of the form \\
\begin{align} \label{homeq1stepc}
\{\chi_1,  h_0 \} + F_1 &= \la F_1 \ra, 
\end{align}
where $F_1(\psi,\bar\psi) = \frac{\lambda}{16} \int \left[ |\psi+\bar\psi|^2 + |\phi+\bar\phi|^2 \right]^2 \di x$. Hence the transformed Hamiltonian is of the form \\
\begin{align}
H_1(\psi,\phi,\bar\psi,\bar\phi) &= h_0(\psi,\phi,\bar\psi,\bar\phi) 
+ \frac{1}{c^2} \left[ -\frac{1}{2} \left( \la\bar\psi,\Delta\psi\ra + \la\bar\phi,\Delta\phi\ra \right) + \la F_1 \ra(\psi,\phi,\bar\psi,\bar\phi) \right] \nonumber \\
&+ \frac{1}{c^4} \cR^{(1)}(\psi,\phi,\bar\psi,\bar\phi), \label{ham1stepc}
\end{align}
where
\begin{align*}
\la F_1 \ra &= \frac{\lambda}{16} \left[ 6\psi^2 \bar\psi^2 + 6\phi^2 \bar\phi^2 + 8 \psi \bar{\psi} \phi \bar{\phi} + 2 \psi^2 \phi^2 + 2 \bar\psi^2 \bar\phi^2 \right] \\
&= \frac{\lambda}{8} \left[ 3 (|\psi|^2+|\phi|^2)^2 + 2 (\psi\phi - \bar\psi \bar\phi)^2 \right].
\end{align*}
If we neglect the remainder and we derive the corresponding 
equations of motion for the system, we get 
\begin{equation} \label{eqstep1c}
\begin{cases}
 -i \psi_t &= \psi + \frac{1}{c^2} \left\{ -\frac{1}{2} \Delta\psi + \frac{\lambda}{4} \left[ 3(|\psi|^2+|\phi|^2)\psi + 2(\psi\phi+\bar\psi \bar\phi)\bar\phi  \right] \right\}, \\ \\
i \phi_t &= \phi + \frac{1}{c^2} \left\{ -\frac{1}{2} \Delta\phi + \frac{\lambda}{4} \left[ 3(|\psi|^2+|\phi|^2)\phi + 2(\psi\phi+\bar\psi \bar\phi)\bar\psi  \right] \right\}, \\
\end{cases}
\end{equation}
which is a system of two coupled NLS equations.
\end{remark}

\section{Dynamics} \label{dynamics}

\indent Now we want to exploit the result of the previous section 
in order to deduce some consequences about the dynamics of the NLKG 
equation \eqref{dsa} in the nonrelativistic limit. 
Consider the \emph{simplified system}, that is the Hamiltonian $H_r$ 
in the notations of Theorem \ref{normformgavthm}, where we neglect the 
remainder:
\begin{align*}
H_{simp} &:= h_0+\epsilon(h_1+ \la F_1 \ra)+ \sum_{j=2}^{r} \epsilon^j(h_j+Z_j).
\end{align*}
We recall that in the case of the NLKG the simplified system is 
actually the NLS (given by $h_0+\epsilon(h_1+ \la F_1 \ra)$), 
plus higher-order normalized corrections. Now let $\psi_r$ be a solution of  
\begin{align} \label{simpleq}
-i \,\dot \psi_r \, &= \, X_{H_{simp}}(\psi_r),
\end{align}
then $\psi_a(t,x):=\cT^{(r)}(\psi_r(c^2t,x))$ solves
\begin{align} \label{appreq}
\dot \psi_a&= i c\jap{\grad}_c \psi_a + 
\frac{\lambda}{2l} \left( \frac{c}{\jap{\grad}_c} \right)^{1/2} \, 
\left[ \left( \frac{c}{\jap{\grad}_c} \right)^{1/2} \frac{\psi_a+\bar\psi_a}{\sqrt{2}} \right]^{2l-1}
- \frac{1}{c^{2r}} X_{ \cT^{(r)*}\cR^{(r)} }(\psi_a,\bar\psi_a),
\end{align}
that is, the NLKG plus a remainder of order $c^{-2r}$ (in the following 
we will refer to equation \eqref{appreq} as \emph{approximate equation}, 
and to $\psi_a$ as the \emph{approximate solution} of the original NLKG). 
We point out that the original NLKG and the approximate equation differ 
only by a remainder of order $c^{-2r}$, which is evaluated on the approximate 
solution. This fact is extremely important: indeed, if one can prove the 
smoothness of the approximate solution (which often is easier to check 
than the smoothness of the solution of the original equation), 
then the contribution of the remainder may be considered small 
in the nonrelativistic limit. This property is rather general, 
and has been already applied in the framework of normal form theory 
(see for example \cite{bambusi2002nonlinear}). \\
\indent Now let $\psi$ be a solution of the NLKG equation \eqref{dsa} 
with initial datum $\psi_0$, and let $\delta:=\psi-\psi_a$ be the error 
between the solution of the approximate equation and the original one. 
One can check that $\delta$ fulfills 
\begin{align*}
\dot \delta &= i c \jap{\grad}_c \delta +
[ P(\psi_a+\delta,\bar\psi_a+\bar\delta)-P(\psi_a,\bar\psi_a) ]+
\frac{1}{c^{2r}} X_{ \cT^{(r)*}\cR^{(r)} }(\psi_a(t),\bar\psi_a(t) ),
\end{align*}
where 
\begin{align} \label{nonlin}
P(\psi,\bar\psi)&= \frac{\lambda}{2l} \left( \frac{c}{\jap{\grad}_c} \right)^{1/2}
\left[ \left( \frac{c}{\jap{\grad}_c} \right)^{1/2} \frac{\psi+\bar\psi}{\sqrt{2}} \right]^{2l-1}. 
\end{align}
Thus we get
\begin{align}
\dot \delta &= i \,c \jap{\grad}_c\delta + dP(\psi_a(t))\delta + \cO(\delta^2) + \cO\left(\frac{1}{ c^{2r} }\right); \nonumber \\
\delta(t)&= e^{itc\jap{\grad}_c}\delta_0 + 
\int_0^{t}e^{i(t-s)c\jap{\grad}_c}dP(\psi_a(s))\delta(s)\di s +
\cO(\delta^2)+\cO\left(\frac{1}{c^{2r}}\right). \label{erreqrem}
\end{align}
By applying Gronwall inequality to \eqref{erreqrem} we obtain 

\begin{proposition} \label{locuniftconv}
Fix $r \geq 1$, $R>0$, $k_1 \gg 1$, $1<p< +\infty$. 
Then $\exists$ $k_0=k_0(r)>0$ with the following properties: 
for any $k \geq k_1$ there exists $c_{l,r,k,p,R} \gg 1$ such that 
for any $c>c_{l,r,k,p,R}$, if we assume that
\begin{align*}
\|\psi_0\|_{k+k_0,p} &\leq R
\end{align*}
and that there exists $T=T_{r,k,p}>0$ such that 
the solution of \eqref{simpleq} satisfies
\begin{align*}
\|\psi_r(t)\|_{k+k_0,p} & \leq 2R, \; \; \text{for} \; \; 0 \leq t \leq T,
\end{align*}
then
\begin{align} \label{esterrprop}
\|\delta(t)\|_{k,p} &\leq C_{k,p} \, c^{-2r},\; \; \text{for} \; \; 0\leq t \leq T.
\end{align}
\end{proposition}

\begin{remark}
If we restrict to $p=2$, and to $M=\T^d$, the above result is 
actually a reformulation of Theorem 3.2 in \cite{faou2014asymptotic}. 
We also remark that the time interval $[0,T]$ in which estimate 
\eqref{esterrprop} is valid is \emph{independent} of $c$.
\end{remark}

\begin{remark}
By exploiting estimate \eqref{CTthm} about the canonical transformation, 
Proposition \ref{locuniftconv} leads immediately to a proof of 
Theorem \ref{introlocuniftconv}.
\end{remark}

In order to study the evolution of the error between the 
approximate solution and the solution of the NLKG over longer 
(namely, $c$-dependent) time scales, we observe that the error is described by
\begin{align} \label{erreq}
\dot\delta(t)&= i \, c\jap{\grad}_c\delta(t) + dP(\psi_a(t))\delta(t); \\
\delta(t) &= e^{itc\jap{\grad}_c}\delta_0+\int_0^{t}e^{i(t-s)c\jap{\grad}_c}dP(\psi_a(s))\delta(s)\di s,
\end{align}
up to a remainder which is small, if we assume the smoothness of $\psi_a$.

Equation \eqref{erreq} in the context of dispersive PDEs is known as
\emph{semirelativistic spinless Salpeter equation} with a time-dependent 
potential. This system was introduced as a first order in time analogue of 
the KG equation for the Lorentz-covariant description 
of bound states within the framework of relativistic quantum field theory, 
and, despite the nonlocality of its Hamiltonian, some of its properties 
have already been studied 
(see \cite{sucher1963relativistic} for a study from a physical point of view; 
for a more mathematical approach see \cite{lammerzahl1993pseudodifferential} 
and the more recent works \cite{carles2012higher} and \cite{carles2015higher}, 
which are closer to the spirit of our approximation). 

It seems reasonable to estimate the solution of Equation \eqref{erreq} 
by studying and by exploiting its dispersive properties, and this 
will be the aim of the following sections. From now on we will consider 
by simplicity only the three-dimensional case, $d=3$, 
but the argument may also be applied to $M=\R^d$ for $d \geq 2$. 

\section{Long time approximation} \label{longtappr}

Now we study the evolution of the the error between the approximate solution 
$\psi_a$, namely the solution of \eqref{appreq}, and the original solution 
$\psi$ of \eqref{dsa} for long (that means, $c$-dependent) time intervals. 
As pointed out in Sect. \ref{results}, we will prove a result only 
for the linear case; we will also begin to discuss the long time approximation 
of the NLKG, but we defer more precise results to a future work. \\

\subsection{Linear case} \label{lincase}

\indent Fix $r \geq 1$, and take $\psi_0 \in H^{k+k_0}$, where $k_0>0$ and 
$k \gg 1$ are the ones in Theorem \ref{normformgavthm}. 
In \cite{carles2012higher} and \cite{carles2015higher} the authors proved 
that the linearized normal form system, namely the one that corresponds 
(up to a rescaling of time by a factor $c^2$) to 
\begin{align} 
-i \dot{\psi_r} &= X_{h_0 + \sum_{j=1}^r \epsilon^j h_j}(\psi_r), \label{schrordr} \\
\psi_r(0) &= \psi_0, \nonumber
\end{align}
admits a unique solution in $L^\infty(\R)H^{k+k_0}(\R^3)$  
(this is a simple application of the properties of the Fourier transform), and 
by a perturbative argument they also proved the global existence also for the 
higher oder Schr\"odinger equation with a bounded time-independent potential.

Moreover, by following the arguments of Theorem 4.1 in \cite{kim2012global} 
and Lemma 4.3 in \cite{carles2015higher} one obtains the following 
dispersive estimates and local-in-time Strichartz estimates 
for solutions of the linearized normal form equation \eqref{schrordr}.

\begin{proposition}
Let $r \geq 1$, and denote by $\cU_r(t)$ the evolution operator of 
\eqref{schrordr}. Then one has the following local-in-time dispersive estimate 
\begin{align} \label{locdispest}
\| \cU_r(t) \|_{L^1(\R^3) \to L^\infty(\R^3)} &\sleq 
c^{3 \left( 1-\frac{1}{r} \right)} |t|^{-3/(2r)}, \; \; 0<|t| \leq c^{2(r-1)}.
\end{align}
On the other hand, $\cU_r(t)$ is unitary on $L^2(\R^3)$. \\
Now introduce the following set of admissible exponent pairs:
\begin{align} \label{Deltar}
\Delta_r &:= \left\{ (p,q): (1/p,1/q) \; \text{lies in the closed quadrilateral ABCD}\right\},
\end{align}
where 
\[ A=\left(\frac{1}{2},\frac{1}{2}\right), \; \; 
B=\left(1,\frac{1}{\tau_r}\right), \; \; C=(1,0), \; \;
D=\left(\frac{1}{\tau_r'},0\right), \; \; \tau_r = \frac{2r-1}{r-1}, \; \; 
\frac{1}{\tau_r} + \frac{1}{\tau_r'} = 1. \] 
Then for any $(p,q)\in\Delta_r \setminus \{(2,2),(1,\tau_r),(\tau_r',\infty)\}$ 
\begin{align} \label{LpLqhighschr}
\| \cU_r(t) \|_{L^p(\R^3) \to L^q(\R^3)} &\sleq 
c^{3 \left( 1-\frac{1}{r} \right) \left( \frac{1}{p}-\frac{1}{q} \right)} |t|^{-\frac{3}{2r} \left( \frac{1}{q}-\frac{1}{p} \right) }, \; \; 0<|t| \leq c^{2(r-1)}.
\end{align}
\end{proposition}

\begin{figure}[!h]
\includegraphics[width = \textwidth]{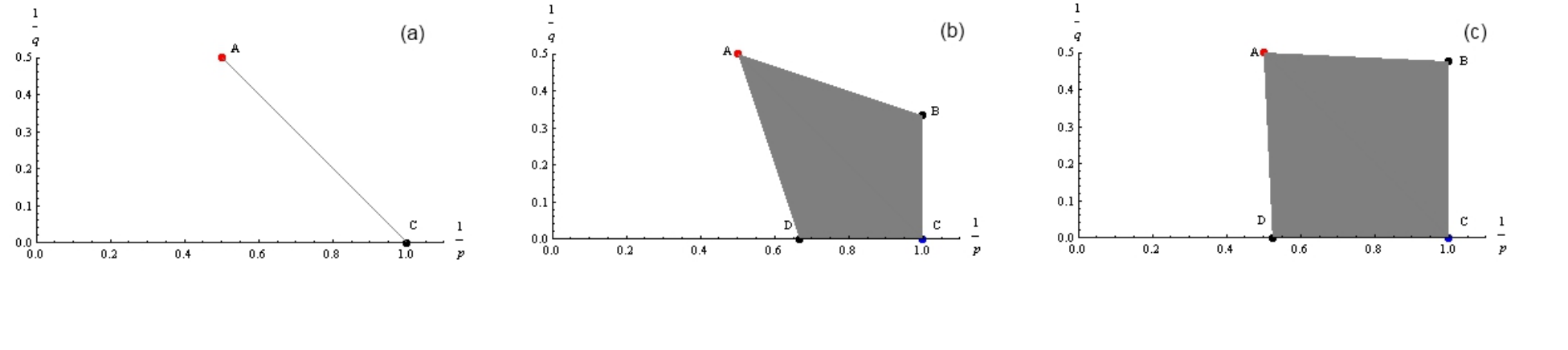}
\caption{Set of admissible exponents $\Delta_r$ for different values of r: (a) r=1 (this is the Schr\"odinger case); (b) r=2; (c) r=11. }
\label{Deltarfig}
\end{figure}

Let $r \geq 1$: in the following lemma $(p,q)$ is called an order-$r$ 
admissible pair when $2 \leq q \leq +\infty$ for 
$r \geq 2$ ($2 \leq q \leq 6$ for $r=1$), and
\begin{align} \label{admhighschr}
\frac{2}{p} + \frac{3}{rq} &= \frac{3}{2r}.
\end{align}

\begin{proposition}
Let $r \geq 1$, and denote by $\cU_r(t)$ the evolution operator of 
\eqref{schrordr}. Let $(p,q)$ and $(r,s)$ be order-$r$ admissible pairs, 
then for any $T \sleq c^{2(r-1)}$
\begin{align}
\| \cU_r(t)\phi_0 \|_{L^p([0,T])L^q(\R^3)} &\sleq 
c^{ 3 \left( 1-\frac{1}{r} \right) \left(\frac{1}{2} -\frac{1}{q} \right) } \|\phi_0\|_{L^2(\R^3)} 
= c^{ \left(1-\frac{1}{r}\right) \frac{2r}{p} } \|\phi_0\|_{L^2(\R^3)}. \label{strhighschr} 
\end{align}
\end{proposition}

Now, we want to estimate the space-time norm of the error $\delta=\psi-\psi_a$. 
In the linear case we can observe that $\delta$ satisfies
\begin{align}
\dot \delta &= i c \jap{\grad}_c \delta +
\frac{1}{c^{2r}} X_{ \cT^{(r)*}\cR^{(r)} }(\psi_a(t),\bar\psi_a(t) ). \label{erreqlin}
\end{align}

\begin{remark} \label{prooflocunifconv}
By applying the Strichartz estimate \eqref{retstrkg} 
(choose $p=+\infty$, $q=2$, $r=+\infty$, $s=2$), together with estimate 
\eqref{Remthm} for the vector field of the remainder $\cR^{(r)}$, 
estimate \eqref{CTthm} for the canonical transformation $\cT^{(r)}$, 
and estimate \eqref{LpLqhighschr} (choose $p=q=2$), 
we can deduce Theorem \ref{KGtoLSradr}.
\end{remark}

\subsection{The nonlinear case: radiation solutions} \label{radcase}

Now, assume that we want to recover the approach of Sect. \ref{lincase} to 
approximate radiation solutions of the NLKG equation for long ($c$-dependent) 
timescales.

To pursue such a program, even by a perturbative argument, 
we would need to consider a (small) radiation solution $\psi_r=\eta_{rad}$ of 
the normalized system \eqref{simpleq} that exists for all times, 
and such that it satisfies the dispersive estimates \eqref{LpLqhighschr}, 
in order to ensure the approximation up to times of order $\cO(c^{2(r-1)})$. 
However, for $r>1$ the issues of global existence and dispersive estimates 
for \eqref{simpleq} are still open problems, as we point out in the following 
remarks.

\begin{remark} \label{GWPscatrem}
The assumption of global existence for $\psi_r$ is actually a delicate matter. 
Equation \eqref{simpleq} is a nonlinear perturbation of a higher-order 
Schr\"odinger equation.

We recall that in \cite{carles2012higher} and \cite{carles2015higher} 
the authors proved that the linearized system admits a unique solution in 
$L^\infty(\R)H^k(\R^3)$, and 
by a perturbative argument they also proved the global existence also for the 
higher oder Schr\"odinger equation with a bounded time-independent potential.

In the nonlinear case little is known: see for example \cite{cui2007well} 
for the well-posedness for a higher-order nonlinear Schr\"odinger equation.

Even if we restrict to the case $r=2$, the issues of global existence and 
scattering for Eq. \eqref{eqstep2} have not been solved.  
Even though some results for the linearization of Eq. 
\eqref{eqstep2} have already been established (see \cite{ben2000dispersion} and 
\cite{kim2012global} for dispersive estimates, and \cite{carles2015higher} for 
Strichartz estimates), the study of the fourth-order NLS-type (4NLS) equation 
is still open: while there are some papers dealing with the 
local well-posedness of 4NLS (see for example \cite{huo2007refined} for the 
one-dimensional case, \cite{huo2011well} for the multidimensional case), 
global well-posedness and scattering results are much less known. 
The recent \cite{ruzhansky2016global} gives the first global well-posedness 
and scattering result for small radiation solutions of 4NLS in any dimension 
$d \geq 1$, but unfortunately does not cover Eq. \eqref{eqstep2}, 
due to technical reasons. 

We defer a more detailed study of Eq. \eqref{eqstep2} 
(and in general of the normal form equation \eqref{simpleq}), together with 
the approximation of small radiation solutions of for the NLKG on $\R^d$, $d \geq 3$, up to times of order $\cO(c^2)$ (or longer), to a future work.
\end{remark}

\begin{remark}
We point out that the case of the one-dimensional defocusing NLKG is also 
interesting, since for $\lambda=1$ the normalized equation at first step is the 
defocusing NLS, which is integrable. It would be interesting also to understand 
whether globally well-posedness and scattering hold also the normalized order 2 
equation \eqref{eqstep2}, which we later exploit to approximate solutions of 
the NLKG up to times of order $\cO(c^2)$.

Even though there is a one-dimensional integrable 4NLS equation related to the 
dynamics of a vortex filament (see \cite{segata2003well} and references therein), 
\begin{align}
i \psi_t  +  \psi_{xx} + \frac{1}{2} |\psi|^2\psi - \nu \left[
\psi_{xxxx} + \frac{3}{2} |\psi|^2\psi_{xx} + \frac{3}{2} \psi_x^2 \bar\psi + 
\frac{3}{8} |\psi|^4\psi + \frac{1}{2} (|\psi|^2)_{xx}\psi \right] &=0, \; 
\nu \in \R
\end{align}
apparently there is no obvious relation between the above equation and Eq. 
\eqref{eqstep2}. Furthermore, while the issue of local well-posedness for 
one-dimensional fourth-order Nonlinear Schr\"odinger has been quite studied 
(see for example \cite{huo2007refined}), there is only a recent result 
(see \cite{ruzhansky2016global}) about global well-posedness and scattering for 
small radiation solutions of 4NLS, which unfortunately does not cover 
Eq. \eqref{eqstep2}, due to technical reasons.

Therefore it seems difficult to give an explicit condition 
for global well-posedness and scattering for the normal form equation also 
in the one-dimensional case. 
\end{remark}

\subsection{The nonlinear case: standing waves solutions} \label{stwavecase}

Now we consider the approximation of another important type of solutions, the 
so-called standing waves solutions. 

The issue of (in)stability of standing waves and solitons has a long history: 
for the NLS equation and the NLKG the orbital stability of standing waves has 
been discussed first in \cite{shatah1985instability}; for the NLS the orbital stability 
of one soliton solutions has been treated in \cite{grillakis1987stability}, 
while the asymptotic stability has been discussed in \cite{cuccagna2001stabilization} for 
one soliton solutions, and in \cite{rodnianski2005dispersive} and \cite{rodnianski2003asymptotic} for N-solitons. 
For the higher-order Schr\"odinger equation we mention 
\cite{maeda2011existence}, which deals with orbital stability of standing waves 
for fourth-order NLS-type equations. 
For the NLKG equation, the instability of solitons and standing waves has 
been studied in \cite{soffer1999resonances}, \cite{imaikin2006scattering} and \cite{ohta2007strong}.\\

As for the case of radiation solution, we should fix $r \geq 1$, 
and consider a standing wave solution $\psi_r$ of \eqref{simpleq}, 
namely of the form 
\begin{align} \label{stwave}
\psi_r(t,x) &= e^{it\omega} \eta_\omega(x),
\end{align}
where $\omega \in \R$, and  $\eta_\omega \in \cS(\R^3)$
solves
\begin{align*}
-\omega \eta_\omega &= X_{H_{simp}}(\eta_\omega).
\end{align*}

\begin{remark}
Of course the existence of a standing wave for the simplified equation 
\eqref{simpleq} is a far from trivial question (see \cite{grillakis1987stability} for the NLS 
equation, and \cite{maeda2011existence} for the fourth-order NLS-type equation).

For $r=1$ and $\lambda=1$ (namely, the defocusing case), we can exploit 
the criteria in \cite{grillakis1987stability} for existence and stability of 
standing waves for the NLS: we recall that if we fix $\omega>0$ and we consider 
$\eta_\omega$ to be the ground state of the corresponding equation, 
we have that the standing wave solution is orbitally stable for 
$\frac{1}{2} < l < \frac{7}{6}$, 
and unstable for $\frac{7}{6} < l < \frac{5}{2}$. \\
\end{remark}

We also point out that in the case of a standing wave solution, 
if $\delta(t)$ satisfies \eqref{erreq}, then by Duhamel formula
\begin{align*}
\dot\delta &= i c\jap{\grad}_c\delta(t) + \di P(\psi_a(t),\bar\psi_a(t)) \delta(t).
\end{align*}
Since 
\begin{align*}
P(e^{it\omega}\eta_\omega,e^{-it\omega}\bar\eta_\omega) &= 2^{l-1/2} \, \left(\frac{c}{\la\grad\ra_c}\right)^{1/2} \left[ \left(\frac{c}{\la\grad\ra_c}\right)^{1/2} Re(e^{it\omega}\eta_\omega) \right]^{2l-1},
\end{align*}
we have that
\begin{align*}
\di P(\eta_\omega,\bar\eta_\omega)e^{it\omega}h &= 2^{l-1/2} \, \left(\frac{c}{\la\grad\ra_c}\right) \left[ \left(\frac{c}{\la\grad\ra_c}\right)^{1/2} \cos(\omega t) \eta_\omega \right]^{2(l-1)} (e^{it\omega}h + e^{-it\omega}\bar h),
\end{align*}
and by setting $\delta = e^{-it\omega}h$, one gets 
\begin{align}
-i \dot h &= (c \jap{\grad}_c + \omega)h + 2^{l-1/2} \cos^{2(l-1)}(\omega t) \left(\frac{c}{\la\grad\ra_c}\right) \left[ \left(\frac{c}{\la\grad\ra_c}\right)^{1/2} \eta_\omega \right]^{2(l-1)} (h + e^{-2it\omega}\bar h) \label{timedeperreq} \\
&+\left[\di P(\psi_a(s),\bar\psi_a(s))-\di P(\eta_\omega,\bar\eta_\omega)\right] h.
\end{align}

Eq. \eqref{timedeperreq} is a Salpeter spinless equation with a periodic 
time-dependent potential; therefore, in order to get some information about the 
error, one would need the corresponding Strichartz estimates for Eq. 
\eqref{timedeperreq}. Unfortunately, in the literature of dispersive estimates 
there are only few results for PDEs with time-dependent potentials, and the 
majority of them is of perturbative nature; for the Schr\"odinger equation 
we mention \cite{ancona2005some} and \cite{goldberg2009strichartz}, 
in which Strichartz estimates are proved in a non-perturbative framework. \\

\begin{remark}
By using Proposition \ref{locuniftconv} one can show that the NLKG can be 
approximated by the simplified equation \eqref{locuniftconv} locally uniformly 
in time, up to an error of order $\cO(c^{-2r})$.
\end{remark}

\begin{remark}
One could ask whether one could get a similar result for more general 
(in particular, moving) soliton solution of \eqref{simpleq}. 
Apart from the issue  of existence and stability for such solutions, 
one can check that, provided that a moving soliton solution for \eqref{simpleq} 
exists, then the error $\delta(t)$ must solve a \eqref{timedeperreq}-type 
equation, namely a spinless Salpeter equation with a time-dependent moving 
potential. Unfortunately, since Eq. \eqref{timedeperreq}, unlike KG, is not 
manifestly covariant, one cannot apparently reduce to an analogue equation, 
and once again one cannot justify the approximation over the $\cO(1)$-timescale.
\end{remark}

\begin{appendix}

\section{Proof of Lemma \ref{NFest}} \label{BNFest}

In order to normalize system \eqref{truncsys}, we used an adaptation of Theorem 
4.4 in \cite{bambusi1999nekhoroshev}. The result is based on the method of Lie transform, 
that we will recall in the following. \\

Let $k \geq k_1$ and $p \in (1,+\infty)$ be fixed. \\
Given an auxiliary function $\chi$ analytic on $W^{k,p}$, 
we consider the auxiliary differential equation
\begin{align} \label{auxDE}
\dot \psi &= i\grad_{\bar\psi} \chi(\psi,\bar\psi) =: X_\chi(\psi,\bar\psi)
\end{align}
and denote by $\Phi^t_\chi$ its time-$t$ flow. 
A simple application of Cauchy inequality gives

\begin{lemma} \label{cauchylemma}
Let $\chi$ and its symplectic gradient be analytic in $B_{k,p}(\rho)$. 
Fix $\delta<\rho$, and assume that 
\begin{align*}
\sup_{B_{k,p}(R-\delta)} \|X_\chi(\psi,\bar\psi)\|_{k,p} &\leq \delta.
\end{align*}
Then, if we consider the time-$t$ flow $\Phi^t_\chi$ of $X_\chi$ we have that 
for $|t| \leq 1$ 
\begin{align*}
\sup_{B_{k,p}(R-\delta)} \|\Phi^t_\chi(\psi,\bar\psi)-(\psi,\bar\psi)\|_{k,p} &\leq \sup_{B_{k,p}(R-\delta)} \|X_\chi(\psi,\bar\psi)\|_{k,p}.
\end{align*}
\end{lemma}

\begin{definition}
The map $\Phi := \Phi^1_\chi$ will be called the \emph{Lie transform} 
generated by $\chi$.
\end{definition}

\begin{remark}
Given $G$ analytic on $W^{k,p}$, consider the differential equation
\begin{align} \label{orDE}
\dot \psi &= X_G(\psi,\bar\psi),
\end{align}
where by $X_G$ we denote the vector field of $G$. Now define 
\begin{align*}
\Phi^\ast G(\phi,\bar\phi) &:= G \circ \Phi(\psi,\bar\psi).
\end{align*}
In the new variables $(\phi,\bar\phi)$ defined by 
$(\psi,\bar\psi)=\Phi(\phi,\bar\phi)$ equation \eqref{orDE}
is equivalent to
\begin{align} \label{pullbDE}
\dot \phi &= X_{ \Phi^\ast G }(\phi,\bar\phi).
\end{align}

Using the relation
\begin{align*}
\frac{\di}{\di t} (\Phi^t_\chi)^\ast G &= (\Phi^t_\chi)^\ast \{\chi,G\}, 
\end{align*}

 we formally get
\begin{align} \label{lieseries}
\Phi^\ast G &= \sum_{l=0}^\infty G_l, \\
G_0 &:= G, \\
G_{l} &:= \frac{1}{l} \{\chi,G_{l-1}\}, \; \; l \geq 1.
\end{align}

\end{remark}

In order to estimate the terms appearing in \eqref{lieseries} 
we exploit the following results

\begin{lemma}
Let $R>0$, and assume that $\chi$, $G$ are analytic on $B_{k,p}(R)$. \\
Then, for any $d \in (0,R)$ we have that 
$\{\chi,G\}$ is analytic on $B_{k,p}(R-d)$, and
\begin{align} \label{liebrest}
\sup_{B_{k,p}(R-d)} \|X_{ \{\chi,G\} }(\psi,\bar\psi)\|_{k,p} &\sleq \frac{2}{d}.
\end{align}
\end{lemma}

\begin{lemma}
Let $R>0$, and assume that $\chi$, $G$ are analytic on $B_{k,p}(R)$. 
Let $l \geq 1$, and consider $G_l$ as defined in \eqref{lieseries};
for any $d \in (0,R)$ we have that $G_l$ is analytic on $B_{k,p}(R-d)$, and
\begin{align} \label{lieserest}
\sup_{B_{k,p}(R-d)} \|X_{ G_l }(\psi,\bar\psi)\|_{k,p} &\sleq \left( \frac{2e}{d} \right)^l.
\end{align}
\end{lemma}

\begin{proof}
Fix $l$, and denote $\delta:=d/l$. We look for a sequence $C^{(l)}_m$ such that
\begin{align*}
\sup_{B_{k,p}(R-m\delta)} \|X_{G_m}(\psi,\bar\psi)\|_{k,p} &\sleq C^{(l)}_m, \; \; 
\forall m \leq l.
\end{align*}
By \eqref{liebrest}  we can define the sequence
\begin{align*}
C^{(l)}_0 &:= \sup_{B_{k,p}(R)} \|X_G(\psi,\bar\psi)\|_{k,p}, \\
C^{(l)}_m &= \frac{2}{\delta m} C^{(l)}_{m-1} \, \sup_{B_{k,p}(R)} \|X_\chi(\psi,\bar\psi)\|_{k,p} \\
&= \frac{2l}{dm} \, C^{(l)}_{m-1} \, \sup_{B_{k,p}(R)} \|X_\chi(\psi,\bar\psi)\|_{k,p}.
\end{align*}
One has
\begin{align*}
C^{(l)}_l &= \frac{1}{l!} \left( \frac{2l}{d} \sup_{B_{k,p}(R)} \|X_\chi(\psi,\bar\psi)\|_{k,p} \right)^l  \, \sup_{B_{k,p}(R)} \|X_G(\psi,\bar\psi)\|_{k,p},
\end{align*}
and by using the inequality $l^l < l! e^l$ we can conclude.
\end{proof}

\begin{remark}
Let $k \geq k_1$, $p \in (1,+\infty)$, and assume 
that $\chi$, $F$ are analytic on $B_{k,p}(R)$. 
Fix $d \in (0,R)$, and assume also that
\begin{align*}
 \sup_{B_{k,p}(R)} \|X_\chi(\psi,\bar\psi)\|_{k,p} \leq d/3,
\end{align*}
Then for $|t| \leq 1$
\begin{align}
\sup_{B_{k,p}(R-d)} \|X_{ (\Phi^t_\chi)^\ast F - F }(\psi,\bar\psi)\|_{k,p} &= \sup_{B_{k,p}(R-d)} \|X_{ F \circ \Phi^t_\chi - F }(\psi,\bar\psi)\|_{k,p} \\
&\stackrel{\eqref{liebrest}}{\leq} 
\frac{5}{d} \, \sup_{B_{k,p}(R)} \|X_\chi(\psi,\bar\psi)\|_{k,p} \,
\sup_{B_{k,p}(R)} \|X_F(\psi,\bar\psi)\|_{k,p}. \label{vfest}
\end{align}

\end{remark}

\begin{lemma} \label{homeqlemma}
Let $k \geq k_1$, $p \in (1,+\infty)$, and assume 
that $G$ is analytic on $B_{k,p}(R)$, and that $h_0$ satisfies PER. 
Then there exists $\chi$ analytic on $B_{k,p}(R)$ and $Z$ analytic 
on $B_{k,p}(R)$ with $Z$ in normal form, namely $\{h_0,Z\}=0$, such that
\begin{align} \label{homeq}
\{ h_0 , \chi \} \; + \; G \; &= \; Z.
\end{align}
Furthermore, we have the following estimates on the vector fields
\begin{align} \label{vfhomeq}
\sup_{B_{k,p}(R)} \|X_Z(\psi,\bar\psi)\|_{k,p} &\leq \sup_{B_{k,p}(R)} \|X_G(\psi,\bar\psi)\|_{k,p}, \\
\sup_{B_{k,p}(R)} \|X_\chi(\psi,\bar\psi)\|_{k,p} &\sleq \sup_{B_{k,p}(R)} \|X_G(\psi,\bar\psi)\|_{k,p}.
\end{align}
\end{lemma}
\begin{proof}
One can check that the solution of \eqref{homeq} is
\begin{align*}
\chi(\psi,\bar\psi) &= 
\frac{1}{T} \int_0^T t \, \left[ G(\Phi^t(\psi,\bar\psi))-Z(\Phi^t(\psi,\bar\psi)) \right] \di t,
\end{align*}
with $T=2\pi$. Indeed,
\begin{align*}
\{ h_0,\chi \}(\psi,\bar\psi) &= \frac{\di}{\di s}_{|s=0} \chi(\Phi^s(\psi,\bar\psi)) \\
&= \frac{1}{2\pi} \int_0^{2\pi} t \frac{\di}{\di s}_{|s=0} \left[ G(\Phi^{t+s}(\psi,\bar\psi))-Z(\Phi^{t+s}(\psi,\bar\psi)) \right] \di t \\
&= \frac{1}{2\pi} \int_0^{2\pi} t \frac{\di}{\di t} \left[ G(\Phi^{t}(\psi,\bar\psi))-Z(\Phi^{t}(\psi,\bar\psi)) \right] \di t \\
&= \frac{1}{2\pi} \left[ t G(\Phi^{t}(\psi,\bar\psi))- t Z(\Phi^{t}(\psi,\bar\psi)) \right]_{t=0}^{2\pi} - \frac{1}{2\pi} \int_0^{2\pi} \left[ G(\Phi^{t}(\psi,\bar\psi))-Z(\Phi^{t}(\psi,\bar\psi)) \right] \di t \\
&= G(\psi,\bar\psi)-Z(\psi,\bar\psi).
\end{align*}
Finally, \eqref{vfhomeq} follows from the fact that
\begin{align*}
X_\chi(\psi,\bar\psi) &= 
\frac{1}{T} \int_0^T t \, \Phi^{-t} \circ X_{G-Z}(\Phi^t(\psi,\bar\psi) \di t
\end{align*}
by applying property \eqref{compnorms}.
\end{proof}

\begin{lemma}
Let $k \geq k_1$, $p \in (1,+\infty)$, and assume 
that $G$ is analytic on $B_{k,p}(R)$, and that $h_0$ satisfies PER. 
Let $\chi$ be analytic on $B_{k,p}(R)$, and assume that it solves \eqref{homeq}. 
For any $l\geq 1$ denote by $h_{0,l}$ 
the functions defined recursively as in \eqref{lieseries} from $h_0$.
Then for any $d \in (0,R)$ one has that $h_{0,l}$ is analytic on $B_{k,p}(R-d)$, 
and
\begin{align} \label{lieseriesh0}
\sup_{B_{k,p}(R-d)} \|X_{ h_{0,l} }(\psi,\bar\psi)\|_{k,p} &\leq 
2 \sup_{B_{k,p}(R)} \|X_G(\psi,\bar\psi)\|_{k,p} 
\left( \frac{5}{d} \, \sup_{B_{k,p}(R)} \|X_\chi(\psi,\bar\psi)\|_{k,p} \right)^l.
\end{align}
\end{lemma}

\begin{proof}
By using \eqref{homeq} one gets that $h_{0,1} = Z-G$ is analytic on $B_{k,p}(R)$. 
Then by exploiting \eqref{vfest} one gets the result.
\end{proof}

\begin{lemma} \label{itlemma}
Let $k_1 \gg 1$, $p \in (1,+\infty)$, $R>0$, $m\geq 0$, 
and consider the Hamiltonian 
\begin{align} \label{Hm}
H^{(m)}(\psi,\bar\psi) &= h_0(\psi,\bar\psi) + \epsilon \hat h(\psi,\bar\psi) + 
\epsilon Z^{(m)}(\psi,\bar\psi) + \epsilon^{m+1} F^{(m)}(\psi,\bar\psi).
\end{align}
Assume that $h_0$ satisfies PER and INV, that $\hat h$ satisfies NF, and that
\begin{align*}
\sup_{B_{k,p}(R)} \|X_{\hat h}(\psi,\bar\psi)\|_{k,p} &\leq F_0, \\
\sup_{B_{k,p}(R)} \|X_{F^{(0)}}(\psi,\bar\psi)\|_{k,p} &\leq F.
\end{align*}
Fix $\delta < R/(m+1)$, and assume also that $Z^{(m)}$ are analytic on 
$B_{k,p}(R-m\delta)$, and that
\begin{align}
\sup_{B_{k,p}(R-m\delta)} \|X_{Z^{(0)}}(\psi,\bar\psi)\|_{k,p} &=0, \nonumber \\
\sup_{B_{k,p}(R-m\delta)} \|X_{Z^{(m)}}(\psi,\bar\psi)\|_{k,p} &\leq F\sum_{i=0}^{m-1}\epsilon^i K_s^i, \; \; m \geq 1, \nonumber \\
\sup_{B_{k,p}(R-m\delta)} \|X_{F^{(m)}}(\psi,\bar\psi)\|_{k,p} &\leq F \, K_s^m, \; \; m \geq 1, \label{stepm}
\end{align}
with $K_s:= \frac{2\pi}{\delta} (18F+5F_0)$. \\
Then, if $\epsilon K_s < 1/2$ there exists a canonical transformation 
$\cT^{(m)}_\epsilon$ analytic on $B_{k,p}(R-(m+1)\delta)$ such that 
\begin{align} \label{CTm}
\sup_{B_{k,p}(R-m\delta)} \| \cT^{(m)}_\epsilon(\psi,\bar\psi)-(\psi,\bar\psi) \|_{k,p} &\leq 2\pi \epsilon^{m+1} F,
\end{align}
$H^{(m+1)} := H^{(m)} \circ \cT^{(m)}$ has the form \eqref{Hm} and 
satisfies \eqref{stepm} with $m$ replaced by $m+1$.
\end{lemma}

\begin{proof}
The key point of the lemma is to look for $\cT^{(m)}_\epsilon$ 
as the time-one map of the Hamiltonian vector field of an analytic function 
$\epsilon^{m+1}\chi_m$. Hence, consider the differential equation
\begin{align} \label{chim}
(\dot\psi,\dot{\bar\psi}) &= X_{\epsilon^{m+1}\chi_m}(\psi,\bar\psi);
\end{align}
by standard theory we have that, if $\|X_{\epsilon^{m+1}\chi_m}\|_{B_{k,p}(R-m\delta)}$ 
is sufficiently small and $(\psi_0,\bar\psi_0) \in B_{k,p}(R-(m+1)\delta)$, 
then the solution of \eqref{chim} exists for $|t| \leq 1$. 
Therefore we can define $\cT^t_{m,\epsilon}:B_{k,p}(R-(m+1)\delta) \to B_{k,p}(R-m\delta)$, 
and in particular the corresponding time-one map 
$\cT^{(m)}_\epsilon:=\cT^1_{m,\epsilon}$, which is an analytic canonical 
transformation, $\epsilon^{m+1}$-close to the identity. We have 
\begin{align}
& (\cT^{(m+1)}_\epsilon)^\ast \; (h_0 + \epsilon \hat h + \epsilon Z^{(m)} + \epsilon^{m+1} F^{(m)}) = h_0 + \epsilon \hat h + \epsilon Z^{(m)} \nonumber \\
&\; \; \; + \epsilon^{m+1} \left[ \{ \chi_m,h_0 \} + F^{(m)} \right] + \nonumber \\
&\; \; \; + \left( h_0 \circ \cT^{(m+1)} - h_0  - \epsilon^{m+1} \{ \chi_m,h_0 \} \right) + \epsilon (\hat h \circ \cT^{(m+1)} - \hat h) + \epsilon \left( Z^{(m)} \circ \cT^{(m+1)} - Z^{(m)} \right) \label{nonnorm1} \\
&\; \; \; + \epsilon^{m+1} \left( F^{(m)} \circ \cT^{(m+1)} - F^{(m)} \right). \label{nonnorm2}
\end{align}
It is easy to see that the first three terms are already normalized, 
that the term in the second line is the non-normalized part of order m+1 
that will vanish through the choice of a suitable $\chi_m$,
 and that the last lines contains all the terms of order higher than m+1. \\
\indent Now we want to determine $\chi_m$ in order to solve the so-called 
``homological equation''
\begin{align*}
\{ \chi_m,h_0 \} + F^{(m)} \; &= \; Z_{m+1},
\end{align*}
with $Z_{m+1}$ in normal form. The existence of $\chi_m$ and $Z_{m+1}$ 
is ensured by Lemma \ref{homeqlemma}, and by applying \eqref{vfhomeq} 
and the inductive hypothesis we get 
\begin{align} \label{esthomeq}
\sup_{B_{k,p}(R-m\delta)} \|X_{\chi_m}(\psi,\bar\psi)\|_{k,p}&\leq 2\pi F, \\
\sup_{B_{k,p}(R-m\delta)} \|X_{Z_{m+1}}(\psi,\bar\psi)\|_{k,p}&\leq 2\pi F.
\end{align}
\indent Now define $Z^{(m+1)} := Z^{(m)} + \epsilon^{m} \, Z_{m+1}$, and notice that
 by Lemma \ref{cauchylemma} we can deduce the estimate of $X_{Z^{(m+1)}}$ on 
$B_{k,p}(R-(m+1)\delta)$ and \eqref{CTm} at level $m+1$. 
Next, set $\epsilon^{m+2} F^{(m+1)} := \eqref{nonnorm1} + \eqref{nonnorm2}$. 
Then we can use \eqref{vfest} and \eqref{lieseriesh0}, in order to get 
\begin{align}
&\sup_{B_{k,p}(R-(m+1)\delta)} \|X_{ \epsilon^{m+2} F^{(m+1)} }(\psi,\bar\psi)\|_{k,p} \\
&\leq 
\left( \frac{10}{\delta} \epsilon^m K_s^m \, \epsilon F 
+ \frac{5}{\delta} \epsilon F_0 
+ \frac{5}{\delta} \epsilon F \sum_{i=0}^{m-1} \epsilon^i K_s^i 
+ \frac{5}{\delta} \epsilon F \, \epsilon^m K_s^m \right) 
\, \epsilon^{m+1} \sup_{B_{k,p}(R-m\delta)} \|X_{\chi_m}(\psi,\bar\psi)\|_{k,p} 
\nonumber \\
&= \epsilon^{m+2} \left( \frac{10}{\delta} \epsilon^m K_s^m \, F 
+ \frac{5}{\delta} F_0 
+ \frac{5}{\delta} F \sum_{i=0}^{m-1} \epsilon^i K_s^i 
+ \frac{5}{\delta} F \, \epsilon^m K_s^m \right) 
\, \sup_{B_{k,p}(R-m\delta)} \|X_{\chi_m}(\psi,\bar\psi)\|_{k,p}. \label{remest}
\end{align}

If $m=0$, then the third term is not present, and \eqref{remest} reads
\begin{align*}
\sup_{B_{k,p}(R-\delta)} \|X_{ \epsilon^2 F^{(1)} }(\psi,\bar\psi)\|_{k,p} &\leq 
 \epsilon^{2} \left( \frac{15}{\delta} \, F + \frac{5}{\delta} F \, \right) 
\, 2\pi F < \epsilon^2 K_s F.
\end{align*}
If $m \geq 1$, we exploit the smallness condition $\epsilon K_s < 1/2$, and 
\eqref{remest} reads
\begin{align*}
\sup_{B_{k,p}(R-(m+1)\delta)} \|X_{ \epsilon^{m+2} F^{(m+1)} }(\psi,\bar\psi)\|_{k,p} &< 
\left( \frac{18}{\delta} \epsilon F + \frac{5}{\delta} \epsilon F_0 \right) 2\pi \, \epsilon F \, \epsilon^m K_s^m 
= \epsilon^{m+2} \, F K_s^{m+1}.
\end{align*}

\end{proof}

Now fix $R>0$.
\begin{proof} (of Lemma \ref{NFest})
The Hamiltonian \eqref{truncsys} satisfies the assumptions 
of Lemma \ref{itlemma} with $m=0$, $F_{N,r}$ in place of $F^{(0)}$ and 
$h_{N,r}$ in place of $\hat h$, $F = K^{(F,r)}_{k,p} \, r 2^{2Nr}$, 
$F_0 = K^{(h,r)}_{k,p} \, r 2^{2Nr}$ (for simplicity we will continue 
to denote by $F$ and $F_0$ the last two quantities). 
So we apply Lemma \ref{itlemma} with $\delta=R/4$, provided that 
\[ \frac{8\pi}{R} (18 F + 5 F_0) \epsilon < \frac{1}{2}, \]
which is true due to \eqref{smallcond}. 
Hence there exists an analytic canonical transformation 
$\cT^{(1)}_{\epsilon,N}: B_{k,p}(3R/4) \to B_{k,p}(R)$ with
\[ \sup_{B_{k,p}(3R/4)} \|\cT^{(1)}_{\epsilon,N}(\psi,\bar\psi)-(\psi,\bar\psi)\|_{k,p} \leq 2\pi F \, \epsilon, \]
such that 
\begin{align}
& H_{N,r} \circ \cT^{(1)}_{\epsilon,N} = h_0 + \epsilon h_{N,r} + \epsilon Z^{(1)}_N + \epsilon^2 \cR^{(1)}_N, \label{step1} \\
& Z^{(1)}_N := \la F_{N,r} \ra, \\
& \epsilon^2 \cR^{(1)}_N := \epsilon^2 F^{(1)} \nonumber \\
&=\left( h_0 \circ \cT^{(1)}_{\epsilon,N} - h_0  - \epsilon \{ \chi_1,h_0 \} \right) 
+ \epsilon (\hat h_{N,r} \circ \cT^{(1)}_{\epsilon,N} - \hat h_{N,r}) 
+ \epsilon \left( Z^{(1)}_N \circ \cT^{(1)}_{\epsilon,N} - Z^{(1)}_N \right) \nonumber \\
&\; \; + \epsilon^2 \left( F_{N,r} \circ \cT^{(1)}_{\epsilon,N} - F_{N,r} \right), \\
& \sup_{B_{k,p}(3R/4)} \|X_{ h_{N,r}+Z^{(1)}_N }(\psi,\bar\psi)\|_{k,p} \leq F_0+F =: \tilde F_0, \\
& \sup_{B_{k,p}(3R/4)} \|X_{ \cR^{(1)}_N }(\psi,\bar\psi)\|_{k,p} \leq 
\frac{8\pi}{R} (18 F + 5 F_0) F =: \tilde F.
\end{align}
Again \eqref{step1} satisfies the assumptions of Lemma \ref{itlemma} with $m=0$, and $h_{N,r}+Z^{(1)}_N$ and $\cR^{(1)}_N$ in place of $F^{(0)}$ and $\hat h$. \\
Now fix $\delta:=\delta(R)=\frac{R}{4r}$, and apply $r$ times Lemma \ref{itlemma}; we get an Hamiltonian of the form \eqref{stepr}, such that 
\begin{align}
\sup_{B_{k,p}(R/2)} \|X_{ Z^{(r)}_N }(\psi,\bar\psi)\|_{k,p} &\leq  2\tilde F, \\
\sup_{B_{k,p}(R/2)} \|X_{ \cR^{(r)}_N }(\psi,\bar\psi)\|_{k,p} &\leq  \tilde F.
\end{align}
\end{proof}

\section{Interpolation theory for relativistic Sobolev spaces} \label{interp}

In this section we show an analogue of Theorem 6.4.5 (7) in \cite{lofstrm1976interpolation} 
for the relativistic Sobolev spaces $\sWc^{k,p}$, $k \in \R$, $1<p<+\infty$. 
We recall that 
\begin{align*}
\sWc^{k,p}(\R^3) &:= \left\{ u\in L^p: \|u\|_{\sWc^{k,p}}:=\norm{c^{-k} \, \nablac^k  u}_{L^p}<+\infty \right\}, \; \; k \in \R, \; \; 1<p<+\infty.
\end{align*}

In order to state the main result of this section, we exploit 
notations and well known results coming from complex interpolation theory 
(see \cite{lofstrm1976interpolation} for a detailed introduction to this topic).

In order to study the relativistic Sobolev spaces, we have to recall the notion of Fourier multipliers. 

\begin{definition}
Let $1<p<+\infty$, and $\rho \in \cS'$.
We call $\rho$ a \emph{Fourier multiplier} on $L^p(\R^d)$ 
if the convolution $(\cF^{-1}\rho) \ast f \in L^p(\R^d)$ 
for all $f \in L^p(\R^d)$, and if 
\begin{align} \label{normMp}
\sup_{\|f\|_{L^p} =1} \|(\cF^{-1}\rho) \ast f\|_{L^p} &< +\infty.
\end{align}
The linear space of all such $\rho$ is denoted by $M_p$, and is endowed with 
the above norm \eqref{normMp}.
\end{definition}

One can check that for any $p \in (1,+\infty)$ one has $M_p=M_{p'}$ 
(where $1/p + 1/p' =1$), and that by Parseval's formula $M_2 = L^\infty$.
Furthermore, by Riesz-Thorin theorem one gets that for any 
$\rho \in M_{p_0} \cap M_{p_1}$ and for any $\theta \in (0,1)$
\begin{align} \label{interpMp}
\|\rho\|_{M_p} &\leq \|\rho\|_{M_{p_0}}^{1-\theta} \|\rho\|_{M_{p_1}}^\theta, \; \;
\frac{1}{p} = \frac{1-\theta}{p_0} + \frac{\theta}{p_1}.
\end{align}
In particular, one can deduce that $\|\cdot\|_{M_p}$ decreases with 
$p \in (1,2]$, and that $M_p \subset M_q$ for any $1<p<q\leq 2$. \\

More generally, if $H_0$ and $H_1$ are Hilbert spaces, one can introduce a 
similar definition of Fourier multiplier. We use the notation $\cS'(H_0,H_1)$ 
in order to denote the space of all linear continous maps from $\cS(\R^d,H_0)$ to $H_1$.

\begin{definition}
Let $1<p<+\infty$, let $H_0$ and $H_1$ be two Hilbert spaces, 
and consider $\rho \in \cS'(H_0,H_1)$. 
We call $\rho$ a \emph{Fourier multiplier}  
if the convolution $(\cF^{-1}\rho) \ast f \in L^p(H_1)$ 
for all $f \in L^p(H_0)$, and if 
\begin{align} \label{normMpH}
\sup_{\|f\|_{L^p(H_0)} =1} \|(\cF^{-1}\rho) \ast f\|_{L^p(H_1)} &< +\infty.
\end{align}
The linear space of all such $\rho$ is denoted by $M_p(H_0,H_1)$, 
and is endowed with the above norm \eqref{normMpH}.
\end{definition}

Next we recall \emph{Mihlin multipier theorem} (Theorem 6.1.6 in 
\cite{lofstrm1976interpolation}).

\begin{theorem}
Let $H_0$ and $H_1$ be Hilbert spaces, and assume that 
$\rho:\R^d \to L(H_0,H_1)$ be such that
\begin{align*}
|\xi|^\alpha \|D^\alpha\rho(\xi)\|_{L(H_0,H_1)} &\leq K,\; \; 
\forall \xi \in R^d, |\alpha|\leq L
\end{align*}
for some integer $L>d/2$. 
Then $\rho \in M_p(H_0,H_1)$ for any $1<p<+\infty$, and 
\begin{align*}
\|\rho\|_{M_p} &\leq C_p \, K, \; \; 1<p<+\infty.
\end{align*}
\end{theorem}

Now, recall the Littlewood-Paley functions $(\phi_j)_{j\geq0}$ defined in 
\eqref{litpal}, and introduce the maps $\cJ:\cS' \to \cS'$ and 
$\cP:\cS' \to \cS'$ via formulas
\begin{align} \label{retract}
(\cJ f)_j &:= \phi_j \ast f, \; \; j \geq 0, \\
\cP g &:= \sum_{j \geq 0} \tilde\phi_j \ast g_j, \; \; j \geq 0,
\end{align}
where $g=(g_j)_{j \geq 0}$ with $g_j \in \cS'$ for all $j$, and 
\begin{align*}
\tilde\phi_0 &:= \phi_0+\phi_1, \\
\tilde\phi_j &:= \phi_{j-1}+\phi_j+\phi_{j+1}, \; \; j \geq 1.
\end{align*}
One can check that $\cP \circ \cJ f = f$ $\forall f \in \cS'$, since 
$\tilde\phi_j \ast \phi_j = \phi_j$ for all $j$.
We then introduce for $c \geq 1$ and $k \geq 0$ the space 
\begin{align*}
l^{2,k}_c &:= \{ (z_j)_{j \in \Z} \; : \; c^{-k} \sum_{j \in \Z} (c^2+|j|^2)^k |z_j|^2 < +\infty \}.
\end{align*}

\begin{theorem}
Let $c\geq1$, $k\geq0$, $1<p<+\infty$. Then $\jap{\grad}_c^kL^p$ is a retract 
of $L^p(l^{2,k}_c)$, namely that the operators 
\begin{align*}
\cJ: \sWc^{k,p} &\to L^p(l^{2,k}_c) \\
\cP: L^p(l^{2,k}_c) &\to \sWc^{k,p} 
\end{align*}
satisfy $\cP \circ \cJ = id$ on $\sWc^{k,p}$.
\end{theorem}
\begin{proof}
First we show that $\cJ:\sWc^{k,p} \to L^p(l^{2,k}_c)$ is bounded. \\
Since $\cJ f  = (\cF^{-1}\chi_c) \ast \cJ^k_cf$, where
\begin{align*}
(\chi_c(\xi))_j &:= (c^2+|\xi|^2)^{-k/2} \hat\phi_j(\xi), \; \; j\geq 0 \\
\cJ^k_c f &:= \cF^{-1}( (c^2+|\xi|^2)^{k/2} \hat f ),
\end{align*}
we have that for any $\alpha \in \N^d$
\begin{align*}
|\xi|^\alpha \|D^\alpha\chi_c(\xi)\|_{L(\C,l^{2,k}_c)} &\leq 
|\xi|^\alpha \sum_{j\geq0} (2^{jk}c^k |D^\alpha (\chi_c(\xi))_j|) \leq K_\alpha
\end{align*}
because the sum consists of at most two non-zero terms for each $\xi$. 
Thus $\cJ \in M_p( \sWc^{k,p} , L^p(l^{2,k}_c) )$ by  Mihlin 
multiplier Theorem. \\
\indent On the other hand, consider $\cP: L^p(l^{2,k}_c)\to\sWc^{k,p}$. \\
Since $\cJ^k_c \circ \cP g = (\cF^{-1}\delta_c) \ast g_{(k)}$, where
\begin{align*}
g &= (g_j)_{j \geq 0}, \\
g_{(k)} &:= (2^{jk}g_j)_{j \geq 0}, \\
\delta_c(\xi) g &:= \sum_{j\geq0} 2^{-jk} (c^2+|\xi|^2)^{k/2} \tilde\phi_j(\xi) g_j,
\end{align*}
we have that for any $\alpha \in \N^d$
\begin{align*}
|\xi|^\alpha \|D^\alpha\delta_c(\xi)\|_{L(l^{2,k}_c,\C)} &\leq 
|\xi|^\alpha \left[ \sum_{j\geq0} (2^{-jk}c^{-k} |D^\alpha (c^2+|\xi|^2)^{k/2}\tilde\phi_j(\xi)|)^2 \right]^{1/2} \leq K_\alpha,
\end{align*}
because the sum consists of at most four non-zero terms for each $\xi$. 
Thus $\cP \in M_p( L^p(l^{2,k}_c) , \sWc^{k,p} )$ by  Mihlin 
multiplier Theorem, and we can conclude.
\qed \end{proof}

\begin{corollary} \label{interprelsobcor}
Let $\theta \in (0,1)$, and assume that $k_0$, $k_1 \geq 0$ ($k_0\neq k_1$) 
and $p_0$, $p_1 \in (1,+\infty)$ satisfy
\begin{align*}
k &= (1-\theta) k_0 + \theta k_1, \\
\frac{1}{p} &= \frac{1-\theta}{p_0}+\frac{\theta}{p_1}.
\end{align*}
Then 
$( \sWc^{k_0,p}, \sWc^{k_1,p} )_\theta = \sWc^{k,p}$.
\end{corollary}

The previous corollary, combined with the classical \emph{3 lines theorem} 
(Lemma 1.1.2 in \cite{lofstrm1976interpolation}), immediately leads us to the following Proposition.

\begin{proposition} \label{relsobolevinterp}
Let $k_0 \neq k_1$, $1< p <+\infty$, and assume that 
$T:\sWc^{k_0,p} \to \sWc^{k_0,p}$ has norm $M_0$, and 
that $T:\sWc^{k_1,p} \to \sWc^{k_1,p}$ has norm $M_1$. Then 
\begin{align*}
T : \sWc^{k,p} \to \sWc^{k,p}, \; &\; k=(1-\theta)k_0+\theta k_1, \\
\end{align*}
with norm $M \leq M_0^{1-\theta} M_1^\theta$.
\end{proposition}

Now we conclude with the proof of Theorem \ref{strpotthm}.

\begin{proof}[Theorem \ref{strpotthm}]
Estimates \eqref{strpot} clearly follow from Proposition \ref{strlin} if we can 
prove that for any $\alpha$ and for any $q \in [2,6]$ 
\begin{align} 
\|\jap{\grad}_c^\alpha \cW_\pm \jap{\grad}_c^{-\alpha}\|_{L^q \to L^q} &\sleq 1, \\
\|\jap{\grad}_c^\alpha \cZ_\pm \jap{\grad}_c^{-\alpha}\|_{L^q \to L^q} &\sleq 1. \label{boundwaveop}
\end{align}
Indeed in this case one would have 
\begin{align*}
\|\jap{\grad}^{1/q-1/p}_c \, e^{it \cH(x)}P_c(-\Delta+V)\psi_0\|_{L^p_tL^q_x} &= \|\jap{\grad}^{1/q-1/p}_c \, \cW_\pm e^{it \jap{\grad}_c} \cZ_\pm\psi_0\|_{L^p_tL^q_x},
\end{align*}
but 
\begin{align*}
\|\jap{\grad}^{1/q-1/p}_c \, \cW_\pm e^{it \jap{\grad}_c} \cZ_\pm\psi_0\|_{L^q_x} &\sleq \|\jap{\grad}^{1/q-1/p}_c \, e^{it \jap{\grad}_c} \cZ_\pm\psi_0\|_{L^q_x}, 
\end{align*}
hence
\begin{align*}
\|\jap{\grad}^{1/q-1/p}_c \, e^{it \cH(x)}P_c(-\Delta+V)\psi_0\|_{L^p_tL^q_x}  &\sleq c^{\frac{1}{q}-\frac{1}{p}-\frac{1}{2}} \|\jap{\grad}_c^{1/2} \cZ_\pm\psi_0\|_{L^2} \sleq c^{\frac{1}{q}-\frac{1}{p}-\frac{1}{2}} \|\jap{\grad}_c^{1/2} \psi_0\|_{L^2}.
\end{align*}
To prove \eqref{boundwaveop} we first show that it holds for $\alpha=2k$, $k \in \N$. We argue by induction. The case  $k=0$ is true by Theorem \ref{waveopthm}. Now, suppose that \eqref{boundwaveop} holds for $\alpha=2(k-1)$, then
\begin{align*}
&\|(c^2-\Delta)^k\cZ_\pm(c^2-\Delta)^{-k}\|_{L^q \to L^q} = \|(c^2-\Delta)(c^2-\Delta)^{k-1}\cZ_\pm(c^2-\Delta)^{-(k-1)}(c^2-\Delta)^{-1}\|_{L^q \to L^q} \\
&\leq c^2 \|(c^2-\Delta)^{k-1}\cZ_\pm(c^2-\Delta)^{-(k-1)}(c^2-\Delta)^{-1}\|_{L^q \to L^q}\\
&\; \; \; \; \; + \|-\Delta (c^2-\Delta)^{k-1}\cZ_\pm(c^2-\Delta)^{-(k-1)}(c^2-\Delta)^{-1}\|_{L^q \to L^q} \\
&\leq c^2 \|(c^2-\Delta)^{k-1}\cZ_\pm(c^2-\Delta)^{-(k-1)}(c^2-\Delta)^{-1}\|_{L^q \to L^q}\\
&\; \; \; \; \; + \|-\Delta (c^2-\Delta)^{-1} \, (c^2-\Delta)^{k-1}\cZ_\pm(c^2-\Delta)^{-(k-1)}\|_{L^q \to L^q} \\
&+ \|-\Delta (c^2-\Delta)^{k-1} [\cZ_\pm,(c^2-\Delta)^{-1}] (c^2-\Delta)^{-(k-1)}\|_{L^q \to L^q} \\
&\sleq c^2 \|(c^2-\Delta)^{-1}\|_{L^q \to L^q} + \|-\Delta(c^2-\Delta)^{-1}\|_{L^q \to L^q} \sleq 1,
\end{align*}
since 
\begin{align*}
\|[\cZ_\pm,(c^2-\Delta)^{-1}]\|_{L^2 \to L^2} \sleq \frac{|\xi|}{(c^2+|\xi|^2)^2} \leq (c^2+|\xi|^2)^{-3/2}.
\end{align*}

Similarly we can show \eqref{boundwaveop} for $\alpha=-2k$, $k \in \N$. 
By Proposition \ref{relsobolevinterp} one can extend the result to any 
$\alpha \in \R$ via interpolation theory.
\qed \end{proof}

\end{appendix}


\nocite{*} 
\bibliographystyle{alpha}           
\bibliography{P_NLKG_2017a}        

\newcommand{\etalchar}[1]{$^{#1}$}
\begin{thebibliography}{BAKS00}

\bibitem[AC07]{alazard2007semi}
Thomas Alazard and R{\'e}mi Carles.
\newblock {Semi-classical limit of Schr{\"o}dinger--Poisson equations in space
  dimension n $\geq$ 3}.
\newblock {\em Journal of Differential Equations}, 233(1):241--275, 2007.

\bibitem[BAKS00]{ben2000dispersion}
Matania Ben-Artzi, Herbert Koch, and Jean-Claude Saut.
\newblock {Dispersion estimates for fourth order Schr{\"o}dinger equations}.
\newblock {\em Comptes Rendus de l'Acad{\'e}mie des Sciences-Series
  I-Mathematics}, 330(2):87--92, 2000.

\bibitem[Bam99]{bambusi1999nekhoroshev}
Dario Bambusi.
\newblock {Nekhoroshev theorem for small amplitude solutions in nonlinear
  Schr{\"o}dinger equations}.
\newblock {\em Mathematische Zeitschrift}, 230(2):345--387, 1999.

\bibitem[Bam05]{bambusi2005galerkin}
Dario Bambusi.
\newblock {Galerkin averaging method and Poincar{\'e} normal form for some
  quasilinear PDEs}.
\newblock {\em Annali della Scuola Normale Superiore di Pisa-Classe di
  Scienze}, 4(4):669--702, 2005.

\bibitem[BC11]{bambusi2011dispersion}
Dario Bambusi and Scipio Cuccagna.
\newblock {On dispersion of small energy solutions of the nonlinear Klein
  Gordon equation with a potential}.
\newblock {\em American journal of mathematics}, 133(5):1421--1468, 2011.

\bibitem[BCP02]{bambusi2002nonlinear}
Dario Bambusi, Andrea Carati, and Antonio Ponno.
\newblock {The nonlinear Schr{\"o}dinger equation as a resonant normal form}.
\newblock {\em DISCRETE AND CONTINUOUS DYNAMICAL SYSTEMS SERIES B},
  2(1):109--128, 2002.

\bibitem[BD12]{bao2012analysis}
Weizhu Bao and Xuanchun Dong.
\newblock {Analysis and comparison of numerical methods for the Klein--Gordon
  equation in the nonrelativistic limit regime}.
\newblock {\em Numerische Mathematik}, 120(2):189--229, 2012.

\bibitem[BFS16]{baumstark2016uniformly}
Simon Baumstark, Erwan Faou, and Katharina Schratz.
\newblock {Uniformly accurate exponential-type integrators for klein-gordon
  equations with asymptotic convergence to classical splitting schemes in the
  nonlinear schroedinger limit}.
\newblock {\em arXiv preprint arXiv:1606.04652}, 2016.

\bibitem[BGT04]{burq2004strichartz}
Nicolas Burq, Pierre G{\'e}rard, and Nikolay Tzvetkov.
\newblock {Strichartz inequalities and the nonlinear Schr{\"o}dinger equation
  on compact manifolds}.
\newblock {\em American Journal of Mathematics}, 126(3):569--605, 2004.

\bibitem[BL76]{lofstrm1976interpolation}
J{\"o}ran Bergh and Jorgen Lofstrom.
\newblock {Interpolation spaces: an introduction}.
\newblock {\em Springer Verlag, Newe}, 1976.

\bibitem[BMS04]{bechouche2004nonrelativistic}
Philippe Bechouche, Norbert~J Mauser, and Sigmund Selberg.
\newblock {Nonrelativistic limit of Klein-Gordon-Maxwell to
  Schr{\"o}dinger-Poisson}.
\newblock {\em American journal of mathematics}, 126(1):31--64, 2004.

\bibitem[Bou10]{bouclet2010littlewood}
Jean-Marc Bouclet.
\newblock {Littlewood-Paley decompositions on manifolds with ends}.
\newblock {\em Bull. Soc. Math. France}, 138(1):1--37, 2010.

\bibitem[BP06]{bambusi2006metastability}
Dario Bambusi and Antonio Ponno.
\newblock {On metastability in FPU}.
\newblock {\em Communications in mathematical physics}, 264(2):539--561, 2006.

\bibitem[BZ16]{bao2016uniformly}
Weizhu Bao and Xiaofei Zhao.
\newblock {A uniformly accurate (UA) multiscale time integrator Fourier
  pseudospectral method for the Klein--Gordon--Schr{\"o}dinger equations in the
  nonrelativistic limit regime}.
\newblock {\em Numerische Mathematik}, pages 1--41, 2016.

\bibitem[CG07]{cui2007well}
Shangbin Cui and Cuihua Guo.
\newblock {Well-posedness of higher-order nonlinear Schr{\"o}dinger equations
  in Sobolev spaces Hs (Rn) and applications}.
\newblock {\em Nonlinear Analysis: Theory, Methods \& Applications},
  67(3):687--707, 2007.

\bibitem[CLM15]{carles2015higher}
R{\'e}mi Carles, Wolfgang Lucha, and Emmanuel Moulay.
\newblock {Higher-order Schr{\"o}dinger and Hartree--Fock equations}.
\newblock {\em Journal of Mathematical Physics}, 56(12):122301, 2015.

\bibitem[CM12]{carles2012higher}
R{\'e}mi Carles and Emmanuel Moulay.
\newblock {Higher order Schr{\"o}dinger equations}.
\newblock {\em Journal of Physics A: Mathematical and Theoretical},
  45(39):395304, 2012.

\bibitem[CO06]{cho2006semirelativistic}
Yonggeun Cho and Tohru Ozawa.
\newblock {On the semirelativistic Hartree-type equation}.
\newblock {\em SIAM journal on mathematical analysis}, 38(4):1060--1074, 2006.

\bibitem[CS16]{choi2016nonrelativistic}
Woocheol Choi and Jinmyoung Seok.
\newblock {Nonrelativistic limit of standing waves for pseudo-relativistic
  nonlinear Schr{\"o}dinger equations}.
\newblock {\em Journal of Mathematical Physics}, 57(2):021510, 2016.

\bibitem[Cuc01]{cuccagna2001stabilization}
Scipio Cuccagna.
\newblock {Stabilization of solutions to nonlinear Schr{\"o}dinger equations}.
\newblock {\em Communications on Pure and Applied Mathematics},
  54(9):1110--1145, 2001.

\bibitem[CZ11]{cordero2011strichartz}
Elena Cordero and Davide Zucco.
\newblock {Strichartz estimates for the vibrating plate equation}.
\newblock {\em Journal of Evolution Equations}, 11(4):827--845, 2011.

\bibitem[DF08]{d2008strichartz}
Piero D'Ancona and Luca Fanelli.
\newblock {Strichartz and smoothing estimates for dispersive equations with
  magnetic potentials}.
\newblock {\em Communications in Partial Differential Equations},
  33(6):1082--1112, 2008.

\bibitem[DPV05]{ancona2005some}
Piero D'Ancona, Vittoria Pierfelice, and Nicola Visciglia.
\newblock {Some remarks on the Schr{\"o}dinger equation with a potential in
  $L^r_t$ $L^s_x$}.
\newblock {\em Mathematische Annalen}, 333(2):271--290, 2005.

\bibitem[FS14]{faou2014asymptotic}
Erwan Faou and Katharina Schratz.
\newblock {Asymptotic preserving schemes for the Klein--Gordon equation in the
  non-relativistic limit regime}.
\newblock {\em Numerische Mathematik}, 126(3):441--469, 2014.

\bibitem[Gol09]{goldberg2009strichartz}
Michael Goldberg.
\newblock {Strichartz estimates for the Schr{\"o}dinger equation with
  time-periodic $L^{n/2}$ potentials}.
\newblock {\em Journal of Functional Analysis}, 256(3):718--746, 2009.

\bibitem[GSS87]{grillakis1987stability}
Manoussos Grillakis, Jalal Shatah, and Walter Strauss.
\newblock {Stability theory of solitary waves in the presence of symmetry, I}.
\newblock {\em Journal of Functional Analysis}, 74(1):160--197, 1987.

\bibitem[HJ07]{huo2007refined}
Zhaohui Huo and Yueling Jia.
\newblock {A refined well-posedness for the fourth-order nonlinear
  Schr{\"o}dinger equation related to the vortex filament}.
\newblock {\em Communications in Partial Differential Equations},
  32(10):1493--1510, 2007.

\bibitem[HJ11]{huo2011well}
Zhaohui Huo and Yueling Jia.
\newblock {Well-posedness for the fourth-order nonlinear derivative
  Schr{\"o}dinger equation in higher dimension}.
\newblock {\em Journal de math{\'e}matiques pures et appliqu{\'e}es},
  96(2):190--206, 2011.

\bibitem[IKV06]{imaikin2006scattering}
Valery Imaikin, Alexander Komech, and Boris Vainberg.
\newblock {On scattering of solitons for the Klein--Gordon equation coupled to
  a particle}.
\newblock {\em Communications in mathematical physics}, 268(2):321--367, 2006.

\bibitem[KAY12]{kim2012global}
JinMyong Kim, Anton Arnold, and Xiaohua Yao.
\newblock {Global estimates of fundamental solutions for higher-order
  Schr{\"o}dinger equations}.
\newblock {\em Monatshefte f{\"u}r Mathematik}, 168(2):253--266, 2012.

\bibitem[L{\"a}m93]{lammerzahl1993pseudodifferential}
Claus L{\"a}mmerzahl.
\newblock {The pseudodifferential operator square root of the Klein--Gordon
  equation}.
\newblock {\em Journal of mathematical physics}, 34(9):3918--3932, 1993.

\bibitem[LZ16]{lu2016partially}
Yong Lu and Zhifei Zhang.
\newblock {Partially strong transparency conditions and a singular localization
  method in geometric optics}.
\newblock {\em Archive for Rational Mechanics and Analysis}, 222(1):245--283,
  2016.

\bibitem[Mac01]{machihara2001nonrelativistic}
Shuji Machihara.
\newblock {The nonrelativistic limit of the nonlinear Klein-Gordon equation}.
\newblock {\em FUNKCIALAJ EKVACIOJ SERIO INTERNACIA}, 44(2):243--252, 2001.

\bibitem[MN02]{masmoudi2002nonlinear}
Nader Masmoudi and Kenji Nakanishi.
\newblock {From nonlinear Klein-Gordon equation to a system of coupled
  nonlinear Schr{\"o}dinger equations}.
\newblock {\em Mathematische Annalen}, 324(2):359--389, 2002.

\bibitem[MN03]{masmoudi2003nonrelativistic}
Nader Masmoudi and Kenji Nakanishi.
\newblock {Nonrelativistic limit from Maxwell-Klein-Gordon and Maxwell-Dirac to
  Poisson-Schr{\"o}dinger}.
\newblock {\em International Mathematics Research Notices}, 2003(13):697--734,
  2003.

\bibitem[MN08]{masmoudi2008energy}
Nader Masmoudi and Kenji Nakanishi.
\newblock {Energy convergence for singular limits of Zakharov type systems}.
\newblock {\em Inventiones mathematicae}, 172(3):535--583, 2008.

\bibitem[MN10]{masmoudi2010klein}
Nader Masmoudi and Kenji Nakanishi.
\newblock {From the Klein--Gordon--Zakharov system to a singular nonlinear
  Schr{\"o}dinger system}.
\newblock {\em Annales de l'Institut Henri Poincare (C) Non Linear Analysis},
  27(4):1073--1096, 2010.

\bibitem[MNO02]{machihara2002nonrelativistic}
Shuji Machihara, Kenji Nakanishi, and Tohru Ozawa.
\newblock {Nonrelativistic limit in the energy space for nonlinear Klein-Gordon
  equations}.
\newblock {\em Mathematische Annalen}, 322(3):603--621, 2002.

\bibitem[MS11]{maeda2011existence}
Masaya Maeda and Jun-ichi Segata.
\newblock {Existence and stability of standing waves of fourth order nonlinear
  Schr{\"o}dinger type equation related to vortex filament}.
\newblock {\em Funkcialaj Ekvacioj}, 54(1):1--14, 2011.

\bibitem[N{\etalchar{+}}08]{nakanishi2008transfer}
Kenji Nakanishi et~al.
\newblock {Transfer of global wellposedness from nonlinear Klein-Gordon
  equation to nonlinear Schr{\"o}dinger equation}.
\newblock {\em Hokkaido Mathematical Journal}, 37(4):749--771, 2008.

\bibitem[Naj90]{najman1990nonrelativistic}
Branko Najman.
\newblock {The nonrelativistic limit of the nonlinear Klein-Gordon equation}.
\newblock {\em Nonlinear Analysis: Theory, Methods \& Applications},
  15(3):217--228, 1990.

\bibitem[Nak02]{nakanishi2002nonrelativistic}
Kenji Nakanishi.
\newblock {Nonrelativistic limit of scattering theory for nonlinear
  Klein--Gordon equations}.
\newblock {\em Journal of Differential Equations}, 180(2):453--470, 2002.

\bibitem[OT07]{ohta2007strong}
Masahito Ohta and Grozdena Todorova.
\newblock {Strong instability of standing waves for the nonlinear Klein--Gordon
  equation and the Klein--Gordon--Zakharov system}.
\newblock {\em SIAM Journal on Mathematical Analysis}, 38(6):1912--1931, 2007.

\bibitem[PX13]{pausader2013scattering}
Benoit Pausader and Suxia Xia.
\newblock {Scattering theory for the fourth-order Schr{\"o}dinger equation in
  low dimensions}.
\newblock {\em Nonlinearity}, 26(8):2175--2191, 2013.

\bibitem[RSS03]{rodnianski2003asymptotic}
I~Rodnianski, W~Schlag, and A~Soffer.
\newblock {Asymptotic stability of N-soliton states of NLS}.
\newblock {\em arXiv preprint math/0309114}, 2003.

\bibitem[RSS05]{rodnianski2005dispersive}
Igor Rodnianski, Wilhelm Schlag, and Avraham Soffer.
\newblock {Dispersive analysis of charge transfer models}.
\newblock {\em Communications on pure and applied mathematics}, 58(2):149--216,
  2005.

\bibitem[RT87]{robert1987semi}
Didier Robert and Hideo Tamura.
\newblock {Semi-classical estimates for resolvents and asymptotics for total
  scattering cross-sections}.
\newblock {\em Annales de l'IHP Physique th{\'e}orique}, 46(4):415--442, 1987.

\bibitem[RWZ16]{ruzhansky2016global}
Michael Ruzhansky, Baoxiang Wang, and Hua Zhang.
\newblock {Global well-posedness and scattering for the fourth order nonlinear
  Schr{\"o}dinger equations with small data in modulation and Sobolev spaces}.
\newblock {\em Journal de Math{\'e}matiques Pures et Appliqu{\'e}es},
  105(1):31--65, 2016.

\bibitem[S{\etalchar{+}}03]{segata2003well}
J~Segata et~al.
\newblock {Well-posedness for the fourth-order nonlinear Schr{\"o}dinger-type
  equation related to the vortex filament}.
\newblock {\em Differential and Integral Equations}, 16(7):841--864, 2003.

\bibitem[Sch10]{schneider2010bounds}
Guido Schneider.
\newblock {Bounds for the nonlinear Schr{\"o}dinger approximation of the
  Fermi--Pasta--Ulam system}.
\newblock {\em Applicable Analysis}, 89(9):1523--1539, 2010.

\bibitem[SS85]{shatah1985instability}
Jalal Shatah and Walter Strauss.
\newblock {Instability of nonlinear bound states}.
\newblock {\em Communications in Mathematical Physics}, 100(2):173--190, 1985.

\bibitem[SS03]{stein2003complex}
Elias~M Stein and Rami Shakarchi.
\newblock {Complex analysis. Princeton Lectures in Analysis, II}, 2003.

\bibitem[Suc63]{sucher1963relativistic}
J~Sucher.
\newblock {Relativistic Invariance and the Square-Root Klein-Gordon Equation}.
\newblock {\em Journal of Mathematical Physics}, 4(1):17--23, 1963.

\bibitem[SW99]{soffer1999resonances}
A~Soffer and Michael~I Weinstein.
\newblock {Resonances, radiation damping and instabilitym in Hamiltonian
  nonlinear wave equations}.
\newblock {\em Inventiones mathematicae}, 136(1):9--74, 1999.

\bibitem[Tay11]{taylor2011partial}
ME~Taylor.
\newblock {Partial differential equations III. Nonlinear equations. Second.
  vol. 117}.
\newblock {\em Applied Mathematical Sciences. Springer, New York, pp. xxii},
  715, 2011.

\bibitem[Tsu84]{tsutsumi1984nonrelativistic}
Masayoshi Tsutsumi.
\newblock {Nonrelativistic approximation of nonlinear Klein-Gordon equations in
  two space dimensions}.
\newblock {\em Nonlinear Analysis: Theory, Methods \& Applications},
  8(6):637--643, 1984.

\bibitem[Yaj95]{yajima1995w}
Kenji Yajima.
\newblock {The $W^{k,p}$-continuity of wave operators for Schr{\"o}dinger
  operators}.
\newblock {\em Journal of the Mathematical Society of Japan}, 47(3):551--581,
  1995.

\end{thebibliography}

\end{document}